\documentclass{article}
\usepackage{natbib}
\usepackage{amsfonts}
\usepackage{amsmath,amssymb,graphicx,epsfig}
\usepackage{enumerate,amsthm,epstopdf}
\usepackage{authblk}
\usepackage[textsize=footnotesize]{todonotes}
\usepackage{natbib}
\definecolor{bred}{rgb}{0.8,0,0}

\usepackage[textsize=footnotesize]{todonotes}
\newcommand{\cM}{{\hat{M}}}
\newcommand{\E}{\mathbb{E}}

\newcommand{\hl}{h_\lambda}

\newcommand{\rl}{r_\lambda}

\newcommand{\xn}{\bar{x}^\lambda_n}

\newcommand{\x}{\tilde{x}^{X,0,\lambda}}
\newcommand{\V}{\tilde{V}^{X,0,\lambda}}
\newcommand{\z}{\tilde{z}^{X,0,\lambda}}
\newcommand{\eig}{\operatorname{eig}}

\newcommand{\po}{\left(}
\newcommand{\pf}{\right)}

\newcommand{\R}{\mathbb R}
\newcommand{\N}{\mathbb N}

\newcommand{\Id}{I_d}

\newcommand{\nv}[1]{#1}%{{\color{red}#1}}
\renewcommand{\phi}{\varphi}

\newcommand{\pv}{\Phi_V}
\newtheorem{lemma}{Lemma}
\newtheorem{theorem}{Theorem}
\newtheorem{assumption}{Assumption}
\newtheorem{definition}{Definition}
\newtheorem{proposition}{Proposition}
\newtheorem{corollary}{Corollary}
\newtheorem{remark}{Remark}

\begin{document}
\title{Contractive kinetic Langevin samplers\\
beyond global Lipschitz continuity}
% \shorttitle{}
% \author[1,2,3]{Iosif Lytras}\address{University of Edinburgh, Archimedes/Athena Research Centre and National Technical University of Athens}
% \author{Panayotis Mertikopoulos}\address{Univ. Grenoble Alpes, CNRS, Inria, Grenoble INP, LIG, 38000 Grenoble, France}
\author[1,2,4]{Iosif Lytras}
\author[2,3,5]{Panayotis Mertikopoulos}

\affil[1]{University of Edinburgh}
\affil[2]{Archimedes/Athena Research Centre}
\affil[3]{Univ. Grenoble Alpes, CNRS, Inria}
\affil[4]{National Technical Univeristy of Athens}
\affil[5]{National Kapodistrian University of Athens}
% \shortauthors{Lytras and Mertikopoulos}
% \subjclass{65C05, 65C30, 65C40.}
%
% \keywords{Sampling, Kinetic, Contraction, super-linear gradient growth}

\maketitle
\begin{abstract}
    In this paper, we examine the problem of sampling from log-concave distributions with (possibly) superlinear gradient growth under kinetic (underdamped) Langevin algorithms. Using a carefully tailored taming scheme, we propose two novel discretizations of the kinetic Langevin SDE, the tamed exponential euler scheme and tamed OBABO and we show that they are both contractive and satisfy a log-Sobolev inequality. Building on this, we establish a series of non-asymptotic bounds in $2$-Wasserstein and total variation distance between the law reached by each algorithm and the underlying target measure. In particular, we show that both algorithms can reach almost optimal $\tilde{\mathcal{O}}(\epsilon^{-1})$ convergence rate in Wasserstein and Total Variation distances, and $\tilde{\mathcal{O}}(\epsilon^{-\frac{1}{2}})$  under increased Hessian local regularity for the tamed OBABO scheme.
\end{abstract}
\section{Introduction}
% In the area of technological revolution sampling of an unknown distribution is very prominent in many different fields such as machine learning, physics and Bayesian inference.
Consider the problem of sampling from a distribution of the form $\pi \propto \exp (-u)$ where the potential function $u\colon\mathbb{R}^{d} \to \mathbb{R}$ is known but the normalizing constant isn't.
This problem has a wide range of applications in, among others, machine learning, Bayesian inference, physics, and computer science, but the high dimensionality of the sampling space and the geometry of the potential function can make the design of efficient samplers a highly challenging endeavor.

A promising way to overcome (at least partially) the difficulties that arise in this setting is by simulating the associated overdamped Langevin SDE
\begin{equation}\label{eq-over-Langevin}
\begin{aligned}
\mathrm{d} L_t & =-\nabla u\left(L_t\right) \mathrm{d} t+\sqrt{2} \mathrm{~d} B_t, \quad t  \geq 0 \\
L_0 & =\theta_0
\end{aligned}
\end{equation}
and exploiting the fact that this SDE admits the target measure $\exp(- u)$ as its unique invariant measure. Inspired by this observation Langevin-based algorithms have become a popular choice for sampling from distributions in high dimensional spaces.
There has been a lot of work in providing non-asymptotic results for Unadjusted Langevin algorithm (ULA) under various assumptions such as Lipschitz continuity of $\nabla u$ and convexity of $u$. Under the assumption of convexity and gradient Lispchitz continuity important results are obtained in \cite{dalalyan2017theoretical, durmus2017nonasymptotic, durmus2019high, hola, convex}, while in the non-convex case, under convexity at infinity or dissipativity assumptions, one may consult \cite{berkeley} \cite{majka2020nonasymptotic}, \cite{erdogdu2022convergence} for ULA  while for the Stochastic Gradient variant (SGLD) important works are \cite{raginsky}, \cite{nonconvex}, \cite{zhang2023nonasymptotic}.\\ More recently, starting with the work of \cite{vempala2019rapid} important estimates have been obtained under the assumption that the target measure $\pi$ satisfies an isoperimetric inequality and the gradient of $u$ satisfies a global Lipschitz continuity, (\cite{mou2022improved}, \cite{ balasubramanian2022towards}).
The latter assumption has also been relaxed to a weakly smooth assumption (with the gradient still satisfying a linear growth property), see \cite{nguyen2021unadjusted} and \cite{erdogdu2021convergence}.\\
Recent advances on sampling has led to a connection with optimization (\cite{vempala2019rapid},\cite{chewi2024ballistic}) where overdamped Langevin dynamics can be seen as an analogous to gradient descent. It is therefore natural to explore momentum-like algorithms (\cite{nesterov1983method}) for possible improvements. An example of such a method is the discretization of the underdamped/kinetic Langevin SDE given as 
\begin{equation}\label{u-LSDE}
\begin{aligned}
     dV_t&=\left(-\gamma V_t-\nabla u(\theta_t)\right)dt+ \sqrt{2\gamma}dB_t
     \\
    d\theta_t&=V_t
\end{aligned}
\end{equation}
where $\gamma>0$ is the friction coefficient and $B_t$ is a $d$-dimensional Brownian motion.
Similarly to the overdamped Langevin SDE, the discretization of this diffusion can be used as both an MCMC sampler and non-convex optimizer ( \cite{gao2021global}) since under appropriate conditions, the Markov process $\left\{\left(\tilde{\theta}_{t}, \tilde{V}_{t}\right)\right\}_{t  \geq 0}$ has a unique invariant measure given by
\begin{equation}\label{eq-pibeta}
{\Pi}(\mathrm{d} \theta, \mathrm{d} v) \propto \exp \left(-\left(\frac{1}{2}|v|^{2}+u(\theta)\right)\right) \mathrm{d} \theta \mathrm{d} v.
\end{equation}
Consequently, the marginal distribution of \eqref{eq-pibeta} in $\theta$ is precisely the target measure $\pi$. This means that sampling from \eqref{eq-pibeta} in the extended space and then keeping the samples in the $\theta$ space defines a valid sampler for the density $\pi$.
The underdamped Langevin MCMC has shown improved convergence rates in the case of convex and gradient Lipschitz smooth potential $u$, e.g see \cite{cheng2018underdamped},  \cite{dalalyan2018sampling},\cite{monmarche2021high}\cite{altschuler2024faster}, \cite{paulin2026theoretical}. In the non-convex setting important work has been done in \cite{Deniz}, \cite{Gao}, \cite{chau2019stochastic},\cite{10.3150/20-BEJ1297},\cite{zhang2023improved},\cite{schuh2024convergence}, \cite{altschuler2026shifted} while there has been a lot of effort in analyzing the contractive behaviour of these algorithms (\cite{monmarche2021high}, \cite{leimkuhler2024contraction},\cite{schuh2024convergence}).
\\ 
All these works assume that the gradient of the potential $u$ grows at most linearly. The current article aims to extend the landscape and allow for super-linear gradient growth aiming to answer the following questions:
\begin{itemize}
    \item Can we create contractive numerical schemes based on the underdamped Langevin SDE, when the gradient grows superlinearly?
    \item What are the convergence rates of these schemes?
\end{itemize}
The current article manages to answer these questions sufficiently by providing novel insights for the contraction rates of modified ULMC in this ill-conditioned scenario, but also provide state of the art convergence results.
Our contributions can be summarized as follows:
\begin{enumerate}

\item \textbf{Monotonicity-preserving taming scheme}.

We introduce a novel drift regularization $h_\lambda$ which is both globally Lipschitz and strongly monotone, while converging to the true gradient $\nabla u$ as the stepsize $\lambda \rightarrow 0$. This construction extends taming approaches in the sampling literature and provides new technical tools for handling super-linear gradients in kinetic Langevin dynamics.
\item  \textbf{Contractivity of underdamped discretizations}.

Using a weighted Euclidean norm, we prove that both the tamed stochastic exponential scheme and the tamed OBABO scheme are contractive in an almost sure sense. This establishes, for the first time, explicit contraction rates for underdamped numerical schemes beyond the globally Lipschitz regime.
\item \textbf{Log Sobolev inequality of the underdamped discretizations}.\\
We are able to prove that our discretizations satisfy uniform in the number of iterations Log-Soboelv inequalities. These results establish concentration properties  of the algorithms and could possibly have other implications for the use of these algorithms in fields such as differential privacy.
\item \textbf{Non-asymptotic convergence guarantees in Wasserstein distance}.

We derive quantitative bounds in the 2 -Wasserstein distance between the law of the algorithms and the target distribution. Our results yield iteration complexity

$$
\tilde{O}\left(\frac{1}{\varepsilon}\right) 
$$
for the tamed exponential scheme and 
 for the tamed OBABO scheme
to reach an $\varepsilon$-approximation,  which improves upon existing bounds for tamed exponential Euler methods and is new for splitting methods in this regime.
For the second order integrator OBABO we are able to prove an improved complexity \[\tilde{\mathcal{O}} \left(\frac{1}{\epsilon^{\frac{1}{2}}} \right)\] under some extra local Lipschitz Hessian property.
\item \textbf{Non-asymptotic guarantees in Total variation}.\\
We derive non-asymptotic estimate for the Total Variation distance between of the algorithms to the target distribution.\\
These results are new for the underdamped algorithms in the local Lipschitz regime.
We show that to achieve $\epsilon$ accuracy one needs \[\tilde{O}\left(\frac{1}{\varepsilon}\right) \] iterations for the tamed exponential scheme and  for the tamed OBABO scheme.
For the second order integrator OBABO we are able to prove an improved complexity \[\tilde{\mathcal{O}} \left(\frac{1}{\epsilon^{\frac{1}{2}}} \right)\] under some extra local Lipschitz Hessian property.
\end{enumerate}
\textbf{Notation}
Throughout the paper, $d\geq 1$ denotes the dimension and $|\cdot|$ and
$\langle\cdot,\cdot\rangle$ are the Euclidean norm and inner product,
respectively. For matrices, $\|\cdot\|_{\mathrm{op}}$ denotes the operator
norm, and $I_d$ is the $d\times d$ identity matrix. We write $B(x,r)$ for
the Euclidean ball of centre $x$ and radius $r$, $\mathcal N(a,\Sigma)$
for the Gaussian law with mean $a$ and covariance $\Sigma$, and
$\mathcal L(X)$ for the distribution of a random variable $X$.
Expectations and probabilities are denoted by $\mathbb E$ and $\mathbb P$.
We use $z=(x,v)\in\mathbb R^{2d}$ for a position--velocity pair.

The potential is $u:\mathbb R^d\to\mathbb R$ and its gradient is
$h=\nabla u$. The target distribution on position space and the
associated kinetic Gibbs measure are
\[
  \pi(dx)=Z_u^{-1}e^{-u(x)}\,dx,
  \qquad
  \Pi(dx,dv)=\pi(dx)\,\mathcal N(0,I_d)(dv).
\]
The parameters $m>0$, $L>0$, and $l>0$ are those appearing in the
strong-monotonicity and local-Lipschitz assumptions on $h$.
The friction coefficient and discretization step are denoted by
% $\gamma>0$ and $\lambda>0$, respectively. The regularized gradient
% $h_\lambda$ is specified in Definition~1; $r_\lambda$ and
% $R_\lambda$ denote the radii and bounds used in its construction, and
% $M_\lambda$ denotes its global Lipschitz constant. In particular,
% $h_\lambda=h$ on $B(0,r_\lambda-2)$.

For $p\geq 1$, $\mathcal P_p$ denotes the probability measures with finite
$p$-th moment. For probability measures $\mu,\nu$ on a Euclidean space,
$\Gamma(\mu,\nu)$ denotes the set of couplings of $\mu$ and $\nu$.
The $2$-Wasserstein distance is
\[
 W_2(\mu,\nu)
 =\inf_{\zeta\in\Gamma(\mu,\nu)}
 \left(\int|z-z'|^2\,\zeta(dz,dz')\right)^{1/2}.
\]
We also use the modified norm and weighted norm, introduced in
Definitions~2--3,
\[
 \|z\|_M=|M z|,\qquad
 M(x,v)=\left(x,x+\frac{2}{\gamma}v\right),\qquad
 |(x,v)|_{a,b}^2=|x|^2+2b\langle x,v\rangle+a|v|^2,
\]
where $a>b^2$.\\ The distances $W_M$ and $W_{a,b}$ are the corresponding
$2$-Wasserstein distances obtained by replacing $|z-z'|$ with
$\|z-z'\|_M$ and $|z-z'|_{a,b}$, respectively.
When needed, $W_{\psi_2}$ denotes the Orlicz--Wasserstein distance
associated with the sub-Gaussian  norm
\[
 \|X\|_{\psi_2}
 =\inf\bigl\{r>0:\mathbb E[\exp(|X|^2/r^2)]\leq 2\bigr\},
 \qquad
 W_{\psi_2}(\mu,\nu)
 =\inf_{(X,Y)\in\Gamma(\mu,\nu)}\|X-Y\|_{\psi_2}.
\]

For $\mu\ll\nu$, let
\[
 \mathrm{KL}(\mu\|\nu)
 =\int\log\!\left(\frac{d\mu}{d\nu}\right)d\mu,
 \qquad
 \mathcal R_q(\mu\|\nu)
 =\frac{1}{q-1}\log\int
 \left(\frac{d\mu}{d\nu}\right)^q d\nu,\quad q>1.
\]
The total-variation norm is
$\|\mu-\nu\|_{\mathrm{TV}}
 =\sup_{\|f\|_\infty\leq 1}|\int f\,d\mu-\int f\,d\nu|$.

For a measurable map $F$, $F_\#\mu$ is the push-forward of $\mu$ by
$F$. For a Markov kernel $P$, $\mu P$ denotes the probability measure
obtained by applying $P$ to $\mu$, and $P^n$ denotes its $n$-fold
composition. In the OBABO scheme, $P_O$ and $P_V$ denote the
Ornstein--Uhlenbeck and Verlet (BAB) kernels, so that
$P=P_O P_V P_O$.  The notation $\widetilde{\mathcal O}$
suppresses logarithmic factors.
\section{Preliminaries}
\subsection{Blanket assumptions}
Let $h:=\nabla u$. The following assumptions will be in force throughout the article.\\
Without a loss of generality we assume that $u\geq0.$
\begin{assumption}(strong convexity). \label{A1-Ham}
There exists an $m>0$ such that
\[\langle h(x)-h(y),x-y\rangle  \geq m |x-y|^2 \quad  x,y \in \mathbb{R}^d.\]
\end{assumption}

\begin{assumption}\label{A2-Ham}

(Local Lipschitz continuity). There exist $l>0$ and $L>0$ such that
\[|h(x)-h(y)|\leq L (1+|x|+|y|)^l |x-y| \quad  x,y\in \mathbb{R}^d.\]
\end{assumption}
\begin{assumption}\label{A3-Ham}
We have
\[\E |\theta_0|^{2k} + \E |V_0|^{2k}<\infty \quad \forall k \in \mathbb{N}.\]
\end{assumption}
\subsection{Discretizations of the Kinetic Langevin SDE }
% In this part we are going to discuss different numerical schemes that are used in the literature to discretize \eqref{u-LSDE}.
% We present the in the general form with drift coefficient $b$, where in the Lipschitz smooth case the usual use is $b=\nabla u$. In our work we use the drift coefficient $b=\hl$ where $\hl$ will be connected to the original gradient $\nabla u$ and also depend on the stepsize, which we will discuss in a latter section.
We consider numerical schemes for the kinetic (underdamped) Langevin SDE 
\begin{equation}\label{eq:kineticSDE}
\begin{aligned}
dV_t &= -\gamma V_t\, dt - \nabla u(\theta_t)\, dt + \sqrt{2\gamma} dB_t,\\
d\theta_t &= V_t\, dt,
\end{aligned}
\end{equation}
where $\gamma>0$ is the friction coefficient and $B_t$ is a $d$-dimensional Brownian motion.

In the standard case, the drift coefficient is $b = -\nabla u$. In our setting, we will introduce two schemes with modified drifts $b = -h_\lambda$ depending on the step size $\lambda$, which is carefully designed to handle super-linear growth of the gradient.

\subsection{Stochastic exponential scheme}
In this section we propose a modified version of the stochastic exponential scheme, an algorithm which was first developed in \cite{cheng2018underdamped} and analysed in depth in \cite{dalalyan2018sampling} and \cite{gao2021global} under the assumption of Lipschitz continuity for the gradient.
The creation of this algorithm is motivated by the fact that  \eqref{u-LSDE}, after applying It\^{o}'s formula to the product $e^{\gamma t}V_t$ and by following standard calculations, can be rewritten as
\begin{equation}\label{eq-underlangtwo}
\begin{aligned}
  {V}_t&=e^{-\gamma t}V_0 - \int_0^t e^{-\gamma(t-s)}\nabla u({\theta}_s)ds+\sqrt{{2\gamma}} \int_0^t e^{-\gamma(t-s)} dB_s\\
   {\theta_t}&=\theta_0+\int_0^t {V}_s ds.
\end{aligned}
\end{equation}
The stochastic exponential scheme is given by
\begin{equation}\label{eq-tKLMC2}
    \begin{aligned}
    \bar{Y}^\lambda_{n+1}&=  \psi_0(\lambda) \bar{Y}^{\lambda}_n-\psi_1(\lambda) \hl(\xn) +\sqrt{2\gamma } \Xi_{n+1}\\
    \bar{x}^\lambda _{n+1}&=\xn+\psi_1(\lambda)v^{\lambda}_n-\psi_2(\lambda)\hl(\xn)+\sqrt{2\gamma } \Xi'_{n+1}
    \end{aligned}
\end{equation}
with initial condition $(\bar{Y}^\lambda_0,\bar{x}^\lambda_0)=(V_0,\theta_0)$
where $\psi_0(t)=e^{-\gamma t}$ and $\psi_{i+1}=\int_0^t \psi_i(s)ds$. More specifically,\begin{equation}\label{eq-psi}
    \begin{aligned}
        \psi_0(\lambda)&=e^{-\gamma \lambda},\\
        \psi_1(\lambda)&=\frac{1}{\gamma}(1-e^{-\gamma \lambda}),\\
        \psi_2(\lambda)&=\frac{\lambda \gamma + e^{-\gamma \lambda} -1}{\gamma^2}.
    \end{aligned}
\end{equation} 
$\left(\boldsymbol{\Xi}_{k+1}, \boldsymbol{\Xi}_{k+1}^{\prime}\right)$ is a $2 d $-dimensional centered Gaussian vector satisfying the following conditions:
\begin{itemize}
    \item
- $\left(\boldsymbol{\Xi}_{j}, \boldsymbol{\Xi}_{j}^{\prime}\right)^{\prime}$ s are iid and independent of the initial condition $\left({V}_{0}, \theta_{0}\right)$,
\item for any fixed $j$, the random vectors $\left(\left(\boldsymbol{\Xi}_{j}\right)_{1},\left(\boldsymbol{\Xi}_{j}^{\prime}\right)_{1}\right),\left(\left(\boldsymbol{\Xi}_{j}\right)_{2},\left(\boldsymbol{\Xi}_{j}^{\prime}\right)_{2}\right), \ldots,\left(\left(\boldsymbol{\Xi}_{j}\right)_{d},\left(\boldsymbol{\Xi}_{j}^{\prime}\right)_{d}\right)$ are iid with the covariance matrix
$$
\mathbf{C}=\int_{0}^{\lambda}\left[\psi_{0}(t), \psi_{1}(t)\right]^{\top}\left[\psi_{0}(t) , \psi_{1}(t)\right] d t.
$$
\end{itemize}

\noindent At this point and in view of \eqref{eq-underlangtwo}, one claims that the continuous time interpolation of \eqref{eq-tKLMC2} is given by
\begin{equation}\label{eq-contint}
    \begin{aligned}
    \Tilde{Y}_t&=e^{- \gamma(t-n\lambda)}\Tilde{Y}_{n\lambda}-\int_{n\lambda}^t e^{-\gamma(t-s)} \hl(\tilde{x}_{n\lambda}) ds +\sqrt{2\gamma}\int_{n\lambda}^t e^{-\gamma(t-s)}dW_{s}
    \\\Tilde{x}_t&=\tilde{x}_{n\lambda} +\int_{n\lambda}^t \tilde{Y}_{s} ds,
    \end{aligned}
\end{equation}
where $\{W\}_t$ is a d-dimensional Brownian motion.
which naturally leads to the following Lemma.

\begin{lemma}\label{remarkgridpoints}
Let $\lambda >0$, $(v^{\lambda}_n,\xn)$ be given by \eqref{eq-tKLMC2} and $(\tilde{Y}_{n\lambda},\tilde{x}_{n\lambda})$ be given by \eqref{eq-contint}. Then,
 \[\mathcal{L}(\bar{Y}^{\lambda}_n,\xn)=\mathcal{L}(\tilde{Y}_{n\lambda},\tilde{x}_{n\lambda}), \quad \forall n \in \mathbb{N}.\]
 \end{lemma}
 \subsection{OBABO scheme}
 The kinetic Langevin SDE involves two coupled components: a position evolution (driven by velocity) and a velocity evolution (driven by gradient and friction). Direct discretization (Euler or exponential schemes) often leads to stiff systems and degraded contractive properties. Operator splitting, specifically the OBABO sequence, leverages the separable structure of the SDE, alternately updating momentum and position using analytically tractable sub-steps. This approach mirrors successful strategies in other fields, such as Hamiltonian Monte Carlo, for efficiently integrating stiff or high-dimensional dynamics.\\
  Let $\lambda,\gamma>0$ be respectively the stepsize and friction coefficient, and denote $\tilde{\eta} = e^{-\lambda\gamma/2}$. We consider the \nv{time-homogeneous} Markov chain $(x_n,v_n)_{n\in\N}$ on $\R^d\times\R^d$ with transitions given by
\begin{equation}\label{eq:OBABO-chain}
\left\{\begin{array}{rcl}
x_{1} & = & x_0+ \lambda \po \eta v_0 + \sqrt{1-\eta^2} G\pf  - \frac{\lambda^2}2 \hl(x_0)\\
v_1 & = & \eta^2v_0 - \frac{\lambda\tilde{\eta} }{2}\po \hl(x_0) +  \hl(x_1)\pf + \sqrt{1-\eta^2} \po \tilde{\eta} G + G'\pf\,,
\end{array}\right.
\end{equation}
where $G$ and $G'$ are two independent standard (mean $0$ variance $\Id$)  $d$-dimensional Gaussian random variables.
The scheme can be decomposed as the following successive steps:
\begin{align}\label{eq-decomp-OBABO}
v_0'\  & =  \ \tilde{\eta} v_0 + \sqrt{1-\tilde{\eta}^2} G\tag{O}\\
v_{1/2}\ &=  \  v_0' - \frac{\lambda}{2} \hl(x_0) \tag{B}\\
x_1 \ &= \  x_0 + \lambda v_{1/2} \tag{A}\\
v_1'\  & = \ v_{1/2} - \frac{\lambda}{2} \hl(x_1)\tag{B}\\
v_1 \ & = \  \tilde{\eta} v_1' + \sqrt{1-\tilde{\eta}^2} G'\,.\tag{O}
\end{align}

\section{Monotonicity preserving taming scheme}
In order to proceed with our analysis we want to create a drift coefficient that is Lipschitz and strongly monotone, and in order to make our scheme a logical candidate to sample from $\pi$, our drift should converge to $\nabla u$ as $\lambda$ goes to zero.\\
Taming schemes have appeared in numerous works in the sampling literature for example see 
\cite{tula}, \cite{TUSLA}, \cite{lytras2023taming}, \cite{neufeld2022non},\cite{johnston2023kinetic} \cite{pmlr-v258-lytras25a}.\\
The drift coefficient will be designed based on a construction borrowed from the numerics of SDEs literature ( \cite{johnston2024strongly}) , and the preserved monotonicity and Lipschtiz assumption are a major technical contribution compared to these earlier works.
\begin{definition}
We define the following quantities: \[g(x)=h(x)-{m}{x}.\]
 \[\rl=\sqrt{\frac{2}{m}(|\nabla u(0)|^2m^{-1}+d)} +2 +\sqrt{\frac{2}{m}}ln(\frac{1}{\lambda^8})\]
 \[\sup_{B(0,\rl)} |g(x)|\leq |h(0)|+\left[L\left(1+r_\lambda\right)^l+m\right] r_\lambda:=R_\lambda\]
We define \[g_\lambda(x)=t(x)g(x)+R_\lambda s(x) x\]
where 
\[\begin{gathered}
t(x):= \begin{cases}1, & \left|x\right| \leq \rl-1 \\
\rl-\left|x\right|, & \rl-1<\left|x\right|<\rl, \\
0, & \left|x\right|  \geq \rl\end{cases} \\
\end{gathered}\]\\
\[\begin{gathered}
s(x):= \begin{cases}
0, & |x| \leq \rl-2 \\
|x|-\rl+2, & \rl-2<|x|<\rl-1 \\
1, & \rl-1 \leq|x| \leq \rl \\
\frac{\rl}{\left|x\right|}, & |x|  \geq \rl\end{cases}
\end{gathered}\]
Then the tamed drift coefficient is given by $\hl=g_\lambda +m x$
\end{definition}
\begin{lemma}\label{tamed-prop}
The tamed coefficient has the following properties
\begin{enumerate}
\item \[\hl(x)=h(x) \quad \forall x \in B(0,r_\lambda-2)\]
    \item \[|\hl(x)|\leq R_\lambda (1+|x|) + m|x| \quad \forall x \in \mathbb{R}^d\]
    \item \[|\hl(x)-\hl(y)|\leq M_\lambda|x-y|\quad \forall x,y \in \mathbb{R}^d \quad \text{where} \quad M_\lambda:=2R_\lambda r_\lambda +L(1+2r_\lambda) ^l +2m\\\]
    \item \[\langle \hl(x)-\hl(y),x-y\rangle  \geq m |x-y|^2 \quad \forall x,y \in \mathbb{R}^d \]
    \item For a random variable $X$ such that $\mathcal{L}(X)=\pi$,
    \begin{equation}\label{eq-tamingerror}
        \E |h(X)-\hl(X)|^2\leq C_\pi \lambda^4
    \end{equation}
    where $C_\pi\leq \mathcal{O}(d^{2l+1}).$
\end{enumerate}
\end{lemma}
\begin{proof}
    See Proof Section
    \end{proof}
    \begin{remark}
        An important part of this construction is that the new scheme is strongly monotone and Lipschitz, with Lipschitz constant that is logarithmic power.
        This is particularly important and differs from \cite{johnston2024strongly}, as we make use of the sub-gaussianity of $\pi.$
        In the analysis of our algorithms, the taming error appears only under $\pi$, therefore we can have a polylogarithmic Lipschitz constant with respect to $\lambda^{-1}$ which gives a good rate in the taming error.
    \end{remark}
\section{Main results}
In order to prove a contraction for the scheme, we shall use some modified norms which is inspired by standard Lyapunov structures for the underdamped Langevin SDE.
\begin{definition}
    Let $M(x,v)=(x,x+\frac{2}{\gamma}v)$
    We define \[||(x,v)||_M:= ||M(x,v)||\]
    For $\gamma\geq 2,$
    there holds \[\frac{\gamma^{-1}}{\sqrt{3}}\|z\|_2 \leq\|z\|_M \leq \sqrt{3}\|z\|_2\]
    So for the respective Wasserstein distances we have:
\[\frac{\gamma^{-1}}{\sqrt{3}}W_2\leq W_M\leq \sqrt{3}W_2\]
\end{definition}
\begin{definition}
   For $z=(x, v) \in \mathbb{R}^{2 d}$ we introduce the weighted Euclidean norm

$$
|z|_{a, b}^2=|x|^2+2 b\langle x, v\rangle+a|v|^2
$$

for $a, b>0$, which is equivalent to the Euclidean norm on $\mathbb{R}^{2 d}$ as long as $b^2<a$. Under the condition $b^2<a / 4$, we have

\begin{equation}\label{eq-equivdist}
    \frac{1}{2}\min\{1,a\}|z|^2 \leq|z|_{a, b}^2 \leq \frac{3}{2}\max\{1,a\}|z|^2
\end{equation}
So for the respective Wasserstein distances we have:
\[\sqrt{\frac{\min\{1,a\}}{2}}W_2\leq W_{a,b}\leq \sqrt{\frac{3}{2}\max\{1,a\}}W_2.\]
\end{definition}
% Let $a=\frac{1}{M_\lambda}$ where $M_\lambda$ is the Lipschitz constant of $\hl$ and $b=\frac{1}{\gamma}$. 
% For $\gamma \geq 5\sqrt{M_\lambda}$ and $\lambda \leq \frac{1}{2\gamma}$.
\subsection{Results for the Tamed Exponential Euler Scheme}
A central result is the contraction of the algorithm in a modified norm (which also implies contraction in the respective Wasserstein norm). 
\begin{theorem}
Assume
         \begin{equation}
\label{eq:gamma-lambda-condition}
    \gamma\lambda\le \frac12,
    \qquad
    \gamma^2\ge 9\,\frac{M_\lambda^2}{m}.
\end{equation}
     \\
     Let $z_n=(x_n,v_n)$ the $n-th$ iterate of the stochastic exponential tamed scheme with initial condition $z_0$ and $\tilde{z}_n=(\tilde{x}_n,\tilde{v}_n)$ the n-th iterate of the same scheme with initial condition $\tilde{z}_0$ with synchronously coupled noise.
     Then, 
     \[||z_{n+1}-\tilde{z}_{n+1}||^2_M\leq (1-f(\lambda))|z_n-\tilde{z}_n||^2_M \]
     where $f(\lambda)=\frac{m\lambda}{2\gamma}$
     % where $c_0$ is an absolute constant.
     As a result, the tamed euler scheme contracts in $W_M$ with rate $\rho=\sqrt{1-f(\lambda)}$
     \end{theorem}
     We also prove a Log-Sobolev inequality for the algorithm, a property that is deduced from the contractive behaviour of the algorithm. This is a powerful concentration inequality, that implies sub-Gaussian behaviour, and Law of Large numbers for the algorithm (\cite{monmarche2021high}).
     \begin{definition}
    A measure $\mu$ satisfies a log-Sobolev inequality (LSI) with constant
$C_{\mathrm{LSI}}$ if, for all sufficiently regular $f$,
\[
 \operatorname{Ent}_\mu(f^2)
 \leq C_{\mathrm{LSI}}\int|\nabla f|^2\,d\mu,
 \qquad
 \operatorname{Ent}_\mu(g)
 =\int g\log g\,d\mu
 -\left(\int g\,d\mu\right)\log\left(\int g\,d\mu\right).
\]
\end{definition}
     \begin{theorem}
    Let $M(x,v)=(x,x+\frac{2}{\gamma}v).$ Assume that for the initial conditions, $\mathcal{L}(M(\theta_0,v_0))$ satisfy a Log-Sobolev inequality with constant $C_{LSI,0}.$ Then $\mathcal{L}(M(\theta_n,v_n)$ satisfies an LSI with constant
    \[C_{LSI,n}\leq (1-\frac{m}{2\gamma}\lambda)^n C_{LSI,0} + \frac{32}{m}+\mathcal{O}(\lambda \gamma)\]
\end{theorem}
Based on the contractive property of the algorithm
\begin{theorem}
    Let $(\theta_n,V_n)$ the $n$-th iterate of the stocahstic exponential tamed scheme with stepsize $\lambda$. Then, there holds
    \[W_2(\mathcal{L}(\theta_n),\pi)\leq W_2(\mathcal{L}(\theta_n,V_n),\Pi)\leq 2 \sqrt{M_\lambda} (1-f(\lambda))^{n/2}  W_{2}(\mathcal{L}(\theta_0,V_0), \Pi)+ C \sqrt{M_\lambda}\lambda^2 /f(\lambda) \]
    where $C$ depends polynomially on the dimension.
    As a result, to reach $\epsilon$ accuracy, one needs $\mathcal{\tilde{O}}(\epsilon^{-1})$ iterations.
\end{theorem}

\begin{theorem}
    With the sane restrictions as before and assuming that $W_{\psi_2}(\mathcal{L}(\theta_0,V_0),\Pi)<\infty$ there holds
    \[TV(\mathcal{L}(\theta_n,V_n),\Pi)\leq \sqrt{\frac{C_1}{\gamma \lambda^3}}e^{-c_0\lambda \gamma^{-1}n} W_{\psi_2}(\mathcal{L}(\theta_0,V_0),\Pi) + C_2 \sqrt{\lambda^3 \gamma^{-1} n}\]
    for $n\geq \log_{c^{-1}} \left( \frac{W_{\psi_2}(\mathcal{L}(\theta_0,V_0),\Pi)}{\sigma}\right)$ 
    where $c=\sqrt{1- \frac{\lambda m}{2\gamma}}$ and $\sigma=\frac{\gamma \lambda^3}{6}(1-c'\lambda \gamma). $\\
    As a result, one needs $\mathcal{\tilde{O}}(\epsilon^{-{1}})$ iterations  to achieve $\epsilon$ accuracy.
\end{theorem}
\subsection{Results for the Tamed OBABO scheme}
We provide the respective results for the tamed OBABO scheme i.e contraction, Log-Sobolev inequality and convergence in $W_2$ and $TV$.
\begin{theorem}
%    Let $\nv{$\gamma  \geq 2\sqrt{M_{\lambda}}$ and that $\lambda \leq m/(33\gamma^3)$. Set 
% \[a\ =\ 1/M_{\lambda}\,,\qquad  b\ =\ 1/\gamma\,, \qquad \kappa \ = \  m/(3\gamma)\,.\]
Let \begin{equation}\label{eq-restriction-OBABO}
    \cM=M_\lambda^2/m,\quad a=1/\cM,\quad b=1/\gamma,
\quad\gamma\geq2\sqrt{\cM},\quad
 0<\lambda\leq m/(42\gamma^3)\leq 1,\quad \kappa=m/(3\gamma).
\end{equation}
 
Then, $b^2 \leq a/4$ and 
 for all $(x_0,v_0),(y_0,w_0)\in \R^{2d}$, if $(x_n,v_n)$ and $(y_n,w_n)$ is the chain in \eqref{eq:OBABO-chain} started at $(x_0,v_0)$ and $(y_0,w_0)$ respectively then, almost surely, for all $n\in\N$,
\[\|(x_n,v_n)-(y_n,w_n)\|_{a,b}^2 \ \leq \ \po 1 - \lambda \kappa \pf^n \|(x_0,v_0)-(y_0,w_0)\|_{a,b}^2\,.\]
This implies a contraction in $W_{a,b}$ distance with rate $\sqrt{1-\lambda \kappa}$
\end{theorem}
\begin{theorem}\label{th-OBABO-LSI}
Let $S(x,v)=(x+bv,\sqrt{a-b^2}v).$
    Let $z_0=(x_0,v_0)$ be the initial conditions of the algorithm such that $\mathcal{L}(z_0)$ satisfies and LSI with constant $C_{LSI,0}$.\\ 
  Then $\mathcal{L}((x_n,v_n))$ satisfies an LSI with \[C_{LSI,n}\leq ||S^{-1}||^2_{op}\left(||S||^2_{op}(1-\lambda\kappa)^n C_{LSI,0} + \frac{1-r^{2n}}{1-r^{2}}C_1\right)\]
  where $r:=\sqrt{1-\lambda \kappa}$ and 
   \[C_1\leq 2||S||_{op}^2\left(1+\eta^2(1+\lambda^2 M_\lambda / 2)^2\right)(1-\eta^2)\]
\end{theorem}

\begin{theorem}\label{conv-Obabo}
    Let $(x_n,v_n)$ be the $n$-th iterate of the OBABO tamed scheme with stepsize $\lambda$.
    Then, under the restrictions in \eqref{eq-restriction-OBABO},
\[  W_2(\mathcal{L}(x_n,v_n),\Pi)\leq C {M_\lambda}(1-\frac{m}{3\gamma} \lambda)^{n/2} W_2(\mathcal{L}(x_0,v_0),\Pi) + C' M_\lambda^2 \frac{\lambda \gamma}{m^\frac{3}{2}}\]
    where $C'$ depends polynomially on the dimension, \\
    To reach $\epsilon$ accuracy one needs $\mathcal{\tilde{O}}({\epsilon^{-1})}$ iterations.
\end{theorem}
\begin{theorem}
Let the restrictions on $\gamma$ and $\lambda$ as before. Then,
   \[ TV(\mathcal{L}(x_n,v_n),\Pi)\leq C M_\lambda\frac{K}{\lambda^{3 / 2}} (1-\lambda m(3\gamma)^{-1})^{(n-1)/2} W_2(\mu_0,\Pi) +c\lambda^2n,\]
    where $K=\sqrt{\frac{\lambda}{2 \pi\left(1-e^{-\lambda \gamma}\right)}} \max \left(2+3 \lambda^2 M_\lambda / 2,3 \lambda\right).$
    \\ To reach $\epsilon$ accuracy, one needs $\tilde{\mathcal{O}}(\epsilon^{-1})$ iterations.
\end{theorem}
Next, we show that under increased smoothness the tamed OBABO has increased performance as a second-order integrator. We will assume 
\begin{equation}\label{eq-smooth-Hess}
    ||\nabla^2 u(x)-\nabla^2 u(y)||\leq L' (1+|x|+|y|)^{l'}|x-y|
\end{equation}
\begin{theorem}
     Let $(x_n,v_n)$ be the $n$-th iterate of the OBABO tamed scheme with stepsize $\lambda$.
     Assume additionally \eqref{eq-smooth-Hess}.
    Then, we have
\[ W_2(\mathcal{L}(x_n,v_n),\Pi)\leq C \frac{M_\lambda}{\sqrt{m}}(1-\frac{m}{3\gamma} \lambda)^{n/2} W_2(\mathcal{L}(x_0,v_0),\Pi) + C' M_\lambda^2 \frac{\lambda^2 \gamma}{m^\frac{3}{2}}\]
    where $C'$ depends polynomially on the dimension.
    To reach $\epsilon$ accuracy, one needs $\tilde{\mathcal{O}}(\epsilon^{-\frac{1}{2}})$ iterations.
\end{theorem}
\begin{theorem}
Let the restrictions on $\gamma$ and $\lambda$ as before.
  Assume additionally \eqref{eq-smooth-Hess}.
We have
   \[ TV(\mathcal{L}(x_n,v_n),\Pi)\leq C M_\lambda\frac{K}{\lambda^{3 / 2}} (1-\lambda m (3\gamma)^{-1})^{(n-1)/2} W_2(\mu_0,\Pi) +c\lambda^3n,\]
    where $K=\sqrt{\frac{\lambda}{2 \pi\left(1-e^{-\lambda \gamma}\right)}} \max \left(2+3 \lambda^2 M_\lambda / 2,3 \lambda\right).$
    To reach $\epsilon$ accuracy, one needs $\tilde{\mathcal{O}}(\epsilon^{-\frac{1}{2}})$ iterations.
\end{theorem}
\subsection{Comparison with related literature}
The contribution of the current article is twofold. On the one hand, it provides contraction rates for underdamped numerical schemes beyond global Lipschitz continuity of the drift. This is quite novel, as the common way to study these schemes is to study the contraction properties of the continuous type dynamics, which may sometimes be very complicated. For example in \cite{johnston2023kinetic} a very complicated proof was employed which relied on the use of the Moreau-Yosida transform to essentially first sample from the Moreau-Yosida measure and the compare to the original.
On the other hand, this work provides the first convergence results for the Verlet-scheme operator in the super-linearly growing case while also considerably improves the rate of convergence for a tamed exponential euler scheme than the one obtained in \cite{johnston2023kinetic}. In addition, we provide the first convergence rates in the literature (in the local Lipschitz case) in total variation in both the OBABO and the Stochastic Exponential Euler Scheme. Given the
the fact that in many applications the underdamped scheme has been performing better than the overdamped our algorithm can be employed as a viable alternative to other overdamped tamed Langevin algorithms (\cite{tula},\cite{lytras2024tamed},\cite{lytras2023taming},\cite{TUSLA},\cite{neufeld2022non}).
In addition, another novel contribution of our work is the proof of Log-Sobolev inequalities for the iterates of algorithm. There has been a lot of active research in this domain, as such inequalites are quite important to establish results in differential privacy.
\section{Proof Strategy}
\subsection{Contraction rates for different schemes}

% In order to prove contraction rates for different schemes we will adopt the following notation.
% Let $z_n=(x_n,v_n)$ the $n-th$ iterate of the scheme with initial condition $z_0$ and $\tilde{z}_n=(\tilde{x}_n,\tilde{v}_n)$ the n-th iterate of the same scheme with initial condition $\tilde{z}_0$.
% In order to prove a contraction of the form
% \begin{equation}\label{contrgeneral}
% ||z_{n+1}-\tilde{z}_{n+1}||^2_{a,b}\leq (1-f(\lambda))||z_n-\tilde{z}_n||^2_{a,b}    
% \end{equation}
% if one sets $G=\left(\begin{array}{cc}
% I_d & b I_d \\
% b I_d & a I_d
% \end{array}\right)$ and $P$ being the update rule $(z_{n+1}-\tilde{z}_{n+1})=P({z}_{n}-\tilde{z}_n)$ \eqref{contrgeneral} is equivalent to showing \begin{equation}\label{eq-H}
%     \mathcal{H}:=(1-f(\lambda)) G-P^T G P \succ 0
% \end{equation}
% Since $H$ is symmetric it can take the form
% \begin{equation}\label{eq-H-2}
%     \mathcal{H}=\left(\begin{array}{ll}
% A & B \\
% B & C
% \end{array}\right),
% \end{equation}
% a key component in our proofs will be the use of the following theorem:
% \begin{theorem}[\cite{horn2005basic}, Theorem 1.6]\label{theo-schur}
%     Theorem 1.12 Let $A$ be a Hermitian matrix partitioned as

% $$
% A=\left(\begin{array}{ll}
% A_{11} & A_{12} \\
% A_{12}^* & A_{22}
% \end{array}\right)
% $$

% in which $A_{11}$ is square and nonsingular. Then\\
% (a) $A>0$ if and only if both $A_{11}>0$ and $A_{22}-A^{*}_{12}A^{-1}_{11}A_{12}>0$.\\
% (b) $A  \geq 0$ if and only if $A_{11}>0$ and $ A_{22}-A^{*}_{12}A^{-1}_{11}A_{12}  \geq 0$.
% \end{theorem}
Throughout our proofs a prominent role plays the control of the quantity $\hl(x)-\hl(y) \quad \forall x,y \in \mathbb{R}^d$. A key observation is the fact that if $\hl$ is $M_\lambda$-Lipschitz and $m$ strongly monotone (with the constant $M_\lambda$  depending on the stepsize) then one obtains $\hl(x)-\hl(y)=Q (x-y)$ for some matrix (not neccessarily symmetric), where \[Sym(Q)\geq m I_d \quad \text{and} \quad ||Q||_{op}\leq M_\lambda\] It is therefore crucial to have construct a drift coefficient which is global Lipschitz smooth but also strongly monotone. That is the reason that the new tamed scheme, comes into place as it preserves the monotonicity property but is also Lipschitz.
\subsection{Results for the Tamed Exponential scheme}
To control the Wasserstein distance, using the contraction estimates we are able to contol the 'mixing' time of the process to its invariant measure $\Pi_\lambda$, so one needs to bound the $W_2(\pi_\lambda,\Pi).$ where $\pi_\lambda$ is the invariant measure of the algorithm  By exploiting the contraction properties of the alogirthm, it is easy to see that this quantity can be bounded by a one step bias, i.e \[W_2(\Pi P,\Pi)\] which can be easily bounded by Ito's formula (essentially comparing a Langevin SDE with the algorithm with same initial condition for $\lambda$-time horizon).\\
For the total variation, in Corollary \ref{cor-renyi} we establish a key inqeuality for the contraction of Renyi divegence of two processes following the algorithm with different initial conditions, and by Pinsker inequality one can bound $TV(\mu_0 P^n,\Pi P^n)$ for any initial condition.
It remains to bound $TV(\Pi P^n,\Pi)$ which is controlled by the $KL$ diverence of the underdamped Langevin SDE and the algorithm with the same initial condition, controlled via the Girsanov theorem.
\subsection{Results for Tamed OBABO scheme}
To control the Wasserstein distance, using the contraction estimates we are able to contol the 'mixing' time of the process to its invariant measure $\Pi_\lambda$ (which exists and is unique since the algorithm is a contraction), so one needs to bound the $W_2(\Pi_\lambda,\Pi).$ By using the invariance properties of the $P_O$ kernel assosciated with the $O$ step, we are able to show that \[W_{a,b}(\Pi_\lambda,\Pi)\leq C_{M\lambda,\gamma,m}  W_2(\Pi P_V,\Pi)\] in Proposition \ref{prop-base-OBAB} where $P_V$ is the kernel corresponding to the Verlet step of the algorithm. $W_2(\Pi P_V,\Pi)$ is then controlled by comparing $\Pi P_V$ with an appropriate Hamiltonian system (which preserves the invariant measure) started at $\Pi$.\\
For the total variation, key to our results is a 'Harnack-type' inequality for the algorithm, Proposition \ref{prop-TV to Wass} which enables us to encorporate the Wasserstein contraction estimates already proved for the algorithm which leads to a control of $||\mu_0 P^n -\Pi P^n||_{TV} $.\\
Further calculations yield that the remaining quantity to bound, i.e $||\Pi -\Pi P^n||_{TV}$ can be controlled by $d ||\Pi-\Pi P_V||_{TV}$. \\ To bound this quantity, one compares the Verlet step of the algorithm with an appropriate Metropolized system (which preserves the invariant measure) and obtains the result in Lemma \ref{Lemma-comp with metrop}.
\section{Proof section}
\subsection{Proof of properties of taming scheme}
\begin{proof}[Proof of Lemma \ref{tamed-prop}]

\begin{proof}[Proof of (1)]
Noticing that $t(x)=1$ and $s(x)=0 \quad \forall x\in B(0,\rl-2),$ it follows that 
\[h_\lambda(x)=g_\lambda(x)+mx =g(x)+mx=h(x) \quad \forall x\in B(0,\rl-2).\]
    \end{proof}
    \begin{proof}[Proof of (3)]
    We define \[A_1=\bar{B}(0,r_\lambda-2) \quad A_2=\bar{B}(0,r_\lambda-1)/{B}(0,r_\lambda-2) \quad A_3=\bar{B}(0,r_\lambda)/ B_{r_\lambda-1},\quad A_4=B(0,\rl)^c\]
    On $A_1$ there holds since $g_\lambda=g$
    \[|g_\lambda(x)-g_\lambda(y)|\leq Lip_{B_r}|x-y|.\]
    For $x,y \in A_2$ $Lip(s(x))=1$ and $t(x)=1$ on $A_2$
    one obtains
    \[|g_\lambda(x)-g_\lambda(y)|\leq (Lip_{B})_{r}+R_{\rl}\lambda)|x-y|.\]
    For $x,y \in A_3$ since $Lip(t)=1$ and \[|g(x)|\leq R_{\lambda} \quad \forall x \in A_3\] then,
    \[|g_\lambda(x)-g_\lambda(y)|\leq (Lip(g)_{B\rl} +2R_\lambda)|x-y|\]
    For $x,y \in A_4,$ since $|x|,|y| \geq \rl$ which implies that $\frac{\rl^2}{|x||y|}\leq 1,$ we have
    \[\begin{aligned}
        R_\lambda |x-y|^2 -|g_\lambda(x)-g_\lambda(y)|^2 &= R_\lambda^2 |x-y|^2-\big | R_\lambda r_\lambda \frac{x}{|x|}- R_\lambda r_\lambda \frac{y}{|y|}\big |^2
        \\&= R^2_\lambda\left( |x|^2+|y|^2-2\rl+2\langle x,y\rangle(\frac{\rl^2}{|x||y|}-1)\right)
        \\& \geq R^2_\lambda(|x|^2+|y|^2 -2\rl^2 +2\rl^2 -2|x||y|)
        \\& \geq R_\lambda^2 (|x|-|y|)^2
        \\& \geq 0
    \end{aligned}\]
   so for each set $A_i$ we have
   \[|g_\lambda(x)-g_\lambda(y)|\leq (Lip(g)_{B{\rl}}+ r_\lambda R_\lambda)|x-y|\]
   Let $x,y \in \mathbb{R}^d$ arbitrary. It is easy to see that every straight line $l$ intersects the boundary of of any ball at most twice so that there exist at most $l\leq 6$ points where $l$ intersects any $\partial A_i$. We denote $l_1\dots l_{k-1}$ the points of intersection and $l_0=x$ and $l_k=y$. For every line segment $[l_j,l_{j+1}]$ since $A_i$ cover all $\mathbb{R}^d$, every point in $[l_j,l_{j+1}]$ belong at least in one $A_i$. If the two endpoints belong in different sets by the continuity of line segment there would exists a point $x_0\in [l_j,l_{j+1}]$ that belongs in the intersection which is not possible. So both endpoints belong in the same $A_i$ so since $A_i$ is closed the whole segment belongs in an $A_i$. As a result, using a triangle and the Lipschitzness in each $A_i$ one obtains
   \[\begin{aligned}
       |g_\lambda(x)-g_\lambda(y)|&=|g_\lambda(l_0)-g_\lambda(l_k)|
       \\&\leq \sum_{i=1}^k |g(l_{i-1})-g(l_i)|
       \\& \leq  (Lip_{B_r} +R_\lambda r_\lambda) \sum_{i=1}^k|l_i-l_{i-1}|
       \\&=(Lip_{B_r} +R_\lambda r_\lambda) |x-y|
       \end{aligned}\]
       where the last line holds since the segments are collinear.
\end{proof}   
\begin{proof}[Proof of (4)]
  In order to show that $g_\lambda$ is monotone, we will first prove it is monotone in each $A_i$.
Let $x,y\in A_1$. Then, it is trivial due to the monotonicity of $g$. 
Let $x,y\in A_2$. We assume that $|x| \geq |y|$ so $s(x) \geq s(y).$ Then,
\[\begin{aligned}
    \langle g_\lambda(x)-g_\lambda(y),x-y\rangle& = \langle g(x)-g(y),x-y\rangle + R_\lambda \langle s(x)x-s(y)y,x-y\rangle \\& \geq R_\lambda s(x)|x|^2 +s(y)|y|^2 -(s(x)+s(y))\langle x,y\rangle\\& \geq R_\lambda (s(x)|x|-s(y)|y|)(|x|-|y|)\\& \geq 0. 
\end{aligned}\]
For $x,y \in A_3,$
\[\begin{aligned}
    \langle g_\lambda(x)-g_\lambda(y),x-y\rangle &=R_\lambda |x-y|^2 +t(y) \langle g(x)-g(y),x-y\rangle + (t(x)-t(y))\rangle g(x),x-y\rangle\\&  \geq 
    R_\lambda |x-y|^2 -||x|-|y|||g(x)||x-y|\\& \geq 
    (R-|g(x)|)|x-y|^2 \geq 0
\end{aligned}\]
For $x,y\in A_4$,
\[\langle g_\lambda(x)-g_\lambda(y),x-y\rangle=R_\lambda r_\lambda \langle \frac{x}{|x|}-\frac{y}{|y|},x-y\rangle  \geq 0\]
To complete the proof for arbitrary $x,y$ using the $l_i$ notation of the previous proof then $l_i$ are collinear and two consecutive belong to the same $A_i$ so \[x-y=a_i(l_i-l_{i-1}) \quad \text{for some } a_i>0.\]
\[\langle g_\lambda(x)-g_\lambda(y),x-y\rangle =\sum_{i=1}^k a_i\langle g_\lambda(l_i)-g_\lambda(l_{i-1}),l_i-l_{i-1}\rangle  \geq 0.\]
\end{proof}
\begin{proof}[Proof of (5)]
Notice that using Lemma \ref{inv-measure bound} for the first moment of $X$ we have 
\[\E |X|\leq \sqrt{\frac{2}{m}(|\nabla u(0)|m^{-1}+d)}\]
    Since $\pi$ is strongly log-concave it satisfies an LSI with constant $\frac{2}{m}.$
    By Subgaussian concentration of LSI for the 1-Lipschitz function $|\cdot|$ we have
    \[P(|X|\geq r_\lambda-2)\leq P(|X|\geq \E |X| + \sqrt{\frac{2}{m}}ln(\frac{1}{\lambda^8}))\leq e^{-\frac{m}{2} \frac{2}{m} ln^2( \frac{1}{\lambda^8}))}\]
    Writing 
    \[\E |h(X)-\hl(X)|^2\leq \sqrt{4\E |h(X)|^4 +\E |\hl(X)|^4}\sqrt{P(|X|\geq r_\lambda-2)}\]
    Since $\hl(x)\leq R_\lambda (1+|x|) + L|x| $ one obtains 
    \[\E |h(X)-\hl(X)|^2\leq  \sqrt{\E |X|^{4(l+1)}} R_\lambda e^{-ln^2(\frac{1}{\lambda^8})}\leq C \sqrt{\E|X|^{4(l+1)}}\lambda^4 \]
    The result follows by using Lemma \ref{Inv mom}
\end{proof}
\end{proof}
     \subsection{ Proof of Contraction rates of the Monotonic polygonal Stochastic exponential scheme}
     % In the following section we will prove that our scheme has a contractive behaviour a.s in the weighted euclidean norm with $a=\frac{1}{M_{\lambda}}$ , $b=\frac{1}{\gamma}$.\\ 
     % \begin{proposition}
     %       Let $\gamma \geq 5\sqrt{M_{\lambda}}$ and $\lambda\leq \frac{1}{2\gamma}$.
     % \\
     % Let $z_n=(x_n,v_n)$ the $n-th$ iterate of the stochastic exponential tamed scheme with initial condition $z_0$ and $\tilde{z}_n=(\tilde{x}_n,\tilde{v}_n)$ the n-th iterate of the same scheme (with synchronously coupled noise) with initial condition $\tilde{z}_0$.
     % Then, 
     % \[||z_{n+1}-\tilde{z}_{n+1}||^2_{a,b}\leq (1-f(\lambda))|z_n-\tilde{z}_n||^2_{a,b} \]
     % where $f(\lambda)=\frac{m\lambda}{4\gamma}.$
     % \end{proposition}
     We notice that the mean of the next iterate of the tamed exponential euler scheme started at $(x,v)$ is given by
$$
\begin{aligned}
F(x, v):=\left(x+\frac{1-\exp (-\gamma \lambda)}{\gamma} v-\right. & \frac{h-\gamma^{-1}(1-\exp (-\gamma \lambda))}{\gamma} \hl(x), \\
& \left.\exp (-\gamma h) v-\frac{1-\exp (-\gamma h)}{\gamma} \hl(x)\right) .
\end{aligned}
$$
We make a convenient change of coordinates \[(\phi,\psi)=M(x,v)=(x,x+\frac{2}{\gamma}v).\]
In the new coordianates,the mean of the next iteration of the scheme is given by $$\bar{F}=
M \circ F \circ \mathcal{M}^{-1} .
$$
In the new coordinates $$
\left(\phi_{(k+1) h}, \psi_{(k+1) h}\right) \stackrel{d}{=} \bar{F}\left(\phi_{k h}, \psi_{k h}\right)+\mathcal{N}\left(0, \mathcal{M} \Sigma \mathcal{M}^{\top}\right).
$$
We will show that $\bar{F}$ is a contraction, so the law of the next iterates will be a pushforward contractive map convoluted with a normal random variable.

% \begin{lemma}
%     The map $\bar{F}$ is Lipschitz with constant $||\bar{F}||_{lip}\leq1-c_0\frac{m}{\sqrt{\gamma}}\lambda  $ for some $c_0>0$ absolute constant.
% \end{lemma}

 \newcommand{\ip}[2]{\left\langle #1,#2\right\rangle}
\newcommand{\norm}[1]{\left|#1\right|}

\begin{lemma}[Contraction of the transformed exponential Euler mean map]
\label{lem:contraction-transformed-map}
Let $h_\lambda:\R^d\to\R^d$ satisfy
\begin{align}
    \ip{h_\lambda(x)-h_\lambda(y)}{x-y}
    &\ge m\norm{x-y}^2,
    \label{eq:strong-monotonicity}\\
    \norm{h_\lambda(x)-h_\lambda(y)}
    &\le M_\lambda \norm{x-y},
    \label{eq:lipschitz}
\end{align}
for all $x,y\in\R^d$, where $m>0$ and $M_\lambda\ge m$.

Let
\[
    \eta=e^{-\gamma\lambda},
\]
and define $\bar F:\R^d\times\R^d\to\R^d\times\R^d$ by
\begin{equation}
\label{eq:Fbar}
\bar F(\varphi,\psi)
=
\begin{pmatrix}
\displaystyle
\varphi+\frac{1-\eta}{2}(\psi-\varphi)
-\frac{\gamma\lambda-1+\eta}{\gamma^2}h_\lambda(\varphi)
\\[2mm]
\displaystyle
\varphi+\frac{1+\eta}{2}(\psi-\varphi)
-\frac{\gamma\lambda+1-\eta}{\gamma^2}h_\lambda(\varphi)
\end{pmatrix}.
\end{equation}
Assume
\begin{equation}
\label{eq:gamma-lambda-condition}
    \gamma\lambda\le \frac12,
    \qquad
    \gamma^2\ge 9\,\frac{M_\lambda^2}{m}.
\end{equation}
Then
\begin{equation}
\label{eq:contraction-squared}
    \norm{\bar F(z)-\bar F(\widetilde z)}^2
    \le
    \left(1-\frac{m\lambda}{2\gamma}\right)
    \norm{z-\widetilde z}^2,
    \qquad
    z,\widetilde z\in\R^{2d}.
\end{equation}
Equivalently,
\begin{equation}
\label{eq:contraction-lip}
    \|\bar F\|_{\mathrm{Lip}}
    \le
    \sqrt{1-\frac{m\lambda}{2\gamma}}.
\end{equation}
In particular, if
\[
    M(x,v)=\left(x,x+\frac{2}{\gamma}v\right),
    \qquad
    \|(x,v)\|_M:=\norm{M(x,v)},
\]
and $\bar F=M\circ F\circ M^{-1}$, then
\[
    \|F(z)-F(\widetilde z)\|_M^2
    \le
    \left(1-\frac{m\lambda}{2\gamma}\right)
    \|z-\widetilde z\|_M^2.
\]
\end{lemma}

\begin{proof}
Fix
\[
    z=(\varphi,\psi),
    \qquad
    \widetilde z=(\widetilde\varphi,\widetilde\psi),
\]
and set
\begin{equation}
\label{eq:xi-r-q}
    \xi=\varphi-\widetilde\varphi,
    \qquad
    r=(\psi-\varphi)-(\widetilde\psi-\widetilde\varphi),
    \qquad
    q=h_\lambda(\varphi)-h_\lambda(\widetilde\varphi).
\end{equation}
Then
\[
    z-\widetilde z=(\xi,\xi+r),
\]
and hence
\begin{equation}
\label{eq:E}
    E:=\norm{z-\widetilde z}^2
    =\norm{\xi}^2+\norm{\xi+r}^2.
\end{equation}

Write
\[
    u=\gamma\lambda,
    \qquad
    \alpha=\frac{1-\eta}{2},
    \qquad
    \beta=\frac{1+\eta}{2},
\]
and
\[
    c_1=\frac{u-1+\eta}{\gamma^2},
    \qquad
    c_2=\frac{u+1-\eta}{\gamma^2}.
\]
From \eqref{eq:Fbar},
\begin{equation}
\label{eq:diff-Fbar}
    \bar F(z)-\bar F(\widetilde z)
    =
    \begin{pmatrix}
        \xi+\alpha r-c_1q\\
        \xi+\beta r-c_2q
    \end{pmatrix}.
\end{equation}
Therefore, with
\[
    E^+:=\norm{\bar F(z)-\bar F(\widetilde z)}^2,
\]
a direct expansion gives
\begin{align}
E^+-E
={}&
-\frac{1-\eta^2}{2}\norm{r}^2
-\frac{4u}{\gamma^2}\ip{\xi}{q}
\nonumber\\
&\quad
-\frac{2(u+\eta-\eta^2)}{\gamma^2}\ip{r}{q}
+\frac{2\bigl(u^2+(1-\eta)^2\bigr)}{\gamma^4}\norm{q}^2.
\label{eq:exact-energy}
\end{align}
Indeed,
\[
    \alpha+\beta=1,
    \qquad
    \alpha^2+\beta^2-1=-\frac{1-\eta^2}{2},
\]
while
\[
    c_1+c_2=\frac{2u}{\gamma^2},
\]
\[
    \alpha c_1+\beta c_2
    =\frac{u+\eta-\eta^2}{\gamma^2},
\]
and
\[
    c_1^2+c_2^2
    =
    \frac{2\bigl(u^2+(1-\eta)^2\bigr)}{\gamma^4}.
\]

By monotonicity and Lipschitness of $h_\lambda$,
\begin{equation}
\label{eq:q-bounds}
    \ip{\xi}{q}\ge m\norm{\xi}^2,
    \qquad
    \norm{q}\le M_\lambda\norm{\xi}.
\end{equation}
Let
\[
    A:=\frac{1-\eta^2}{2}.
\]
Using Young's inequality,
\begin{align}
\frac{2(u+\eta-\eta^2)}{\gamma^2}
\norm{r}\norm{q}
&\le
\frac{A}{2}\norm{r}^2
+
\frac{2(u+\eta-\eta^2)^2}
     {A\gamma^4}\norm{q}^2.
\label{eq:young}
\end{align}
Combining \eqref{eq:exact-energy}, \eqref{eq:q-bounds}, and
\eqref{eq:young} yields
\begin{align}
E^+-E
\le{}&
-\frac{A}{2}\norm{r}^2
-\frac{4um}{\gamma^2}\norm{\xi}^2
+\frac{K(u)}{\gamma^4}\norm{q}^2,
\label{eq:energy-K}
\end{align}
where
\begin{equation}
\label{eq:Ku}
K(u)
=
2\bigl(u^2+(1-e^{-u})^2\bigr)
+
\frac{4\bigl(u+e^{-u}-e^{-2u}\bigr)^2}
     {1-e^{-2u}}.
\end{equation}

We now use $0\le u\le 1/2$. First,
\[
    1-e^{-u}\le u.
\]
Moreover, since $e^{2u}\ge 1+2u$,
\[
    1-e^{-2u}
    \ge \frac{2u}{1+2u}
    \ge u.
\]
Also,
\[
    u+e^{-u}-e^{-2u}
    =
    u+e^{-u}(1-e^{-u})
    \le 2u.
\]
Consequently,
\begin{align}
K(u)
&\le 2(u^2+u^2)+\frac{4(2u)^2}{u}
\nonumber\\
&=4u^2+16u
\le 18u.
\label{eq:Ku-bound}
\end{align}
Hence, using again $\norm{q}\le M_\lambda\norm{\xi}$,
\begin{align}
E^+-E
\le{}&
-\frac{A}{2}\norm{r}^2
-
\left(
\frac{4um}{\gamma^2}
-
\frac{18uM_\lambda^2}{\gamma^4}
\right)\norm{\xi}^2.
\label{eq:energy-before-gamma}
\end{align}
The condition $\gamma^2\ge 9M_\lambda^2/m$ implies
\[
    \frac{18M_\lambda^2}{\gamma^4}
    \le
    \frac{2m}{\gamma^2},
\]
so
\begin{equation}
\label{eq:energy-after-gamma}
    E^+-E
    \le
    -\frac{A}{2}\norm{r}^2
    -\frac{2um}{\gamma^2}\norm{\xi}^2.
\end{equation}
Since $1-e^{-2u}\ge u$,
\[
    \frac{A}{2}
    =
    \frac{1-e^{-2u}}{4}
    \ge \frac{u}{4},
\]
and therefore
\begin{equation}
\label{eq:energy-diss}
    E^+-E
    \le
    -\frac{u}{4}\norm{r}^2
    -\frac{2um}{\gamma^2}\norm{\xi}^2.
\end{equation}

On the other hand,
\begin{align}
E
&=\norm{\xi}^2+\norm{\xi+r}^2
\nonumber\\
&\le 3\norm{\xi}^2+2\norm{r}^2.
\label{eq:E-upper}
\end{align}
Furthermore, $M_\lambda\ge m$: indeed,
\[
    m\norm{\xi}^2
    \le
    \ip{\xi}{q}
    \le
    \norm{\xi}\norm{q}
    \le
    M_\lambda\norm{\xi}^2.
\]
Thus \eqref{eq:gamma-lambda-condition} implies
\[
    \gamma^2\ge 9m.
\]
In particular,
\[
    \frac{m}{\gamma^2}\le\frac19<\frac14.
\]
It follows that
\begin{align}
\frac{u}{4}\norm{r}^2
+\frac{2um}{\gamma^2}\norm{\xi}^2
&\ge
\frac{um}{2\gamma^2}
\left(2\norm{r}^2+3\norm{\xi}^2\right)
\nonumber\\
&\ge
\frac{um}{2\gamma^2}E,
\label{eq:diss-controls-E}
\end{align}
where the last inequality follows from \eqref{eq:E-upper}. Combining
\eqref{eq:energy-diss} and \eqref{eq:diss-controls-E},
\[
    E^+
    \le
    \left(1-\frac{um}{2\gamma^2}\right)E.
\]
Since $u=\gamma\lambda$,
\[
    E^+
    \le
    \left(1-\frac{m\lambda}{2\gamma}\right)E.
\]
This proves \eqref{eq:contraction-squared}, and taking square roots gives
\eqref{eq:contraction-lip}.
\end{proof}

\newcommand{\pii}{\psi_2}

\begin{proof}[Proof of contraction]
    Since $M$ is a linear map,
    \[||z_{n+1}-z_{n+1}'||_M^2=||M z_{n+1}-M z'_{n+1}||^2\]
    Since we have assumed that the iterates are syncronously coupled, 
    \[  \begin{aligned}||z_{n+1}-z_{n+1}'||_M^2&=||M z_{n+1}-M z'_{n+1}||^2\\&= ||\bar{F }(Mz_n)-\bar{F}(Mz'_n)||^2\\&\leq \left(1-\frac{\lambda  m}{2\gamma}\right) ||Mz_n-Mz'_n||^2 \\&= \left(1-\frac{\lambda  m}{2\gamma}\right) ||z_n-z'_n||^2_M 
    \end{aligned}\]
\end{proof}
\begin{corollary}
If the initial condition of the algorithm satisfies an LSI then $M(x_n,v_n)$ satisfy an LSI. 
\end{corollary}
\begin{proof}
    Since in the new coordinates\[(\phi_{n+1},\psi_{n+1})=\bar{F}(\phi_n,\psi_n)+N(0,\bar{\Sigma})\]
    where the normal distribution is independent, by the behaviour of the log-Sobolev inequality under Lipschitz maps and convolutions using the fact that $\bar{F}$ is a contraction by the previous lemma one obtains
    \[C_{LSI,{n}}\leq (||\bar{F}||^2_{lip})^n C_{LSI,0} + \sum_{k=0}^{n-1} ||\bar{F}||^{2k}_{lip} C_{LSI} (\mathcal{N}(0,{\bar{\Sigma}})).\]
    For the Gaussian writing $\mathcal{M} \Sigma \mathcal{M}^{\top}=\Sigma \otimes I_d$, we can compute

$$
\begin{aligned}
& \bar{\Sigma}_{1,1}=\frac{2 \lambda}{\gamma}-\frac{3}{\gamma^2}+\frac{4 \exp (-\gamma \lambda)}{\gamma^2}-\frac{\exp (-2 \gamma \lambda)}{\gamma^2}=O(\gamma \lambda^3), \\
& \bar{\Sigma}_{1,2}=\frac{2 \lambda}{\gamma}-\frac{1}{\gamma^2}+\frac{\exp (-2 \gamma \lambda)}{\gamma^2}=O(\lambda^2), \\
& \bar{\Sigma}_{2,2}=\frac{2 \lambda}{\gamma}+\frac{5-4 \eta-\eta^2}{\gamma^2} .=\frac{8 \lambda}{\gamma^2}+O(\lambda^2) .
\end{aligned}
$$

We conclude that

$$
\|\bar{\Sigma}\|_{{op}} \leq \frac{8 \lambda}{\gamma}+O(\lambda^2)
$$
As a result, \begin{equation}\label{lsi-gauss}
    C_{{LSI}}(\mathcal{N}(0, \mathcal{M} \Sigma \mathcal{M}^{\top})) \leq \frac{16 \lambda}{\gamma}+O(\lambda^2)
\end{equation}
    Substituting the expressions from the previous lemma and \eqref{lsi-gauss} yields the result.
\end{proof}

\subsection{Proof of Contraction rates for Tamed OBABO scheme}
\newcommand{\sym}{\operatorname{sym}}
\newcommand{\Lip}{\operatorname{Lip}}
\newcommand{\Ent}{\operatorname{Ent}}
\newcommand{\TV}{\operatorname{TV}}
\newcommand{\eps}{\varepsilon}
% \subsection{Preliminary bounds and norm equivalence}
Write $\norm z_0^2=|x|^2+a|v|^2$ and $\norm z_P=\norm z_{a,b}$. Since $b/\sqrt a=\sqrt{\cM}/\gamma\leq1/2$,
\begin{equation}\label{eq:equivalence}
 \frac12\norm z_0^2\leq\norm z_{a,b}^2\leq\frac32\norm z_0^2.
\end{equation}
% These are exactly the comparisons used in equation (12) of the manuscript, with the new $a$. Also,
\begin{equation}\label{eq:small}
 m\leq M_\lambda\leq\cM\leq\frac{\gamma^2}{4},\qquad
 \lambda\gamma\leq\frac1{168}<\frac14,\qquad
 0<\lambda\kappa\leq\frac1{2016}<1.
\end{equation}
We shall use the weaker condition $\lambda\gamma\leq1/4$ in the remainder estimates. The elementary exponential inequalities give
\begin{equation}\label{eq:eta}
 0\leq1-\eta\leq\lambda\gamma/2,\qquad
 0\leq\eta^2-1+\lambda\gamma\leq\lambda^2\gamma^2/2.
\end{equation}

Let
\[
 \bar x=x-y,\quad \bar v=v-w,\quad
 \bar x^+=x^+-y^+,\quad\bar v^+=v^+-w^+,\quad
 \Delta\bar x=\bar x^+-\bar x.
\]
Let $Q,Q'$ with
\[
 h_\lambda(x)-h_\lambda(y)=Q\bar x,\qquad h_\lambda(x^+)-h_\lambda(y^+)=Q'\bar x^+,
\]
where $\sym Q,\sym Q'\succeq mI$ and $\norm Q,\norm{Q'}\leq M_\lambda$. Thus
\begin{equation}\label{eq:delta-x}
 \Delta\bar x=\lambda\eta\bar v-\frac{\lambda^2}{2}Q\bar x.
\end{equation}
Set $D=(Q+Q')/2$, so $\sym D\succeq mI$ and $\norm D\leq M_\lambda$. With $\bar z=(\bar x,\bar v)$,
\begin{equation}\label{eq:decomposition}
 \bar z^+=(I+\lambda A)\bar z+\rho,\qquad
 A=\begin{pmatrix}0&I\\-D&-\gamma I\end{pmatrix},
\end{equation}
where the remainder is exactly
\begin{align}
 \rho_x&=\lambda(\eta-1)\bar v-\frac{\lambda^2}{2}Q\bar x,\label{eq:rho-x}\\
 \rho_v&=\lambda(1-\eta)D\bar x
 +(\eta^2-1+\lambda\gamma)\bar v
 -\frac{\lambda\eta}{2}Q'\Delta\bar x.\label{eq:rho-v}
\end{align}

Expanding the same quadratic norm gives
\begin{align}
 \norm{\bar z^+}_P^2
 &=\norm{\bar z}_P^2+\lambda\bar z^\top(PA+A^\top P)\bar z
   +2\ip{\bar z}{P\rho}+\norm{\lambda A\bar z+\rho}_P^2\notag\\
 &\leq\norm{\bar z}_P^2+\lambda\bar z^\top(PA+A^\top P)\bar z
   +2\norm{\bar z}_P\norm\rho_P
   +2\lambda^2\norm{A\bar z}_P^2+2\norm\rho_P^2.\label{eq:expansion}
\end{align}
For $D=(Q+Q')/2$ 
\[
 A=\begin{pmatrix}0&I\\-D&-\gamma I\end{pmatrix},\qquad
 P=\begin{pmatrix}I&bI\\bI&aI\end{pmatrix},
\]
since $b\gamma=1$, 
\begin{equation}\label{eq:correct-product}
 PA+A^\top P=
 \begin{pmatrix}-b(D+D^\top)&-aD^\top\\-aD&2(b-a\gamma)I\end{pmatrix}.
\end{equation}
% \subsection{Explicit bounds for the same remainder}
Let $X=|\bar x|$ and $V=\sqrt a|\bar v|=|\bar v|/\sqrt\cM$. Using \eqref{eq:eta}, \eqref{eq:delta-x}, and $\norm Q,\norm{Q'},\norm D\leq\cM$, we obtain
\begin{align*}
 |\rho_x|&\leq\frac{\lambda^2\cM}{2}X+
                 \frac{\lambda^2\gamma\sqrt\cM}{2}V,\\
 \sqrt a|\rho_v|&\leq
 \lambda^2\left(\frac{\gamma\sqrt\cM}{2}
                +\frac{\lambda\cM^{3/2}}4\right)X
 +\lambda^2\left(\frac{\gamma^2}{2}+\frac\cM2\right)V.
\end{align*}
Since $\sqrt\cM\leq\gamma/2$ and $\lambda\gamma\leq1/4$, componentwise
\begin{equation}\label{eq:rho-matrix}
 \begin{pmatrix}|\rho_x|\\\sqrt a|\rho_v|\end{pmatrix}
 \leq\lambda^2\gamma^2
 \begin{pmatrix}1/8&1/4\\33/128&5/8\end{pmatrix}
 \begin{pmatrix}X\\V\end{pmatrix}.
\end{equation}
The squared Frobenius norm of the scalar matrix is $8769/16384<1$. Consequently,
\[
 \norm\rho_0\leq\lambda^2\gamma^2\norm{\bar z}_0,
 \qquad
 \norm\rho_P\leq\sqrt3\lambda^2\gamma^2\norm{\bar z}_P
 \leq2\lambda^2\gamma^2\norm{\bar z}_P.
\]
Likewise, in the scaled coordinates $(\bar x,\sqrt a\bar v)$, the two components of $A\bar z$ are bounded by
\[
 \begin{pmatrix}0&\sqrt\cM\\\sqrt\cM&\gamma\end{pmatrix}
 \begin{pmatrix}X\\V\end{pmatrix}.
\]
Its squared Frobenius norm is $2\cM+\gamma^2\leq3\gamma^2/2$. By \eqref{eq:equivalence},
\begin{equation}\label{eq:A-bound}
 \norm{A\bar z}_P^2\leq\frac92\gamma^2\norm{\bar z}_P^2.
\end{equation}
Thus the three error terms in \eqref{eq:expansion} satisfy
\begin{align}
 &2\norm{\bar z}_P\norm\rho_P
 +2\lambda^2\norm{A\bar z}_P^2+2\norm\rho_P^2\notag\\
 &\qquad\leq\bigl(4+9+8(\lambda\gamma)^2\bigr)
                \lambda^2\gamma^2\norm{\bar z}_P^2
 \leq14\lambda^2\gamma^2\norm{\bar z}_P^2.\label{eq:remainder-bound}
\end{align}

% \subsection{The new leading-term estimate}\label{sec:leading}
By \eqref{eq:correct-product},
\begin{align}
 \bar z^\top(PA+A^\top P)\bar z
 &=-2b\ip{\bar x}{D\bar x}
   -2a\ip{\bar v}{D\bar x}
   +2(b-a\gamma)|\bar v|^2\notag\\
 &\leq-\frac{2m}{\gamma}|\bar x|^2
       +2aM_\lambda|\bar x||\bar v|
       +2(\gamma^{-1}-a\gamma)|\bar v|^2.\label{eq:leading-start}
\end{align}
The first term uses the symmetric part only. The mixed term is bounded in operator norm, so its skew contribution is retained. Young's inequality gives
\begin{equation}\label{eq:young}
 2aM_\lambda|\bar x||\bar v|
 \leq\frac m\gamma|\bar x|^2+
       \frac{a^2M_\lambda^2\gamma}{m}|\bar v|^2
 =\frac m\gamma|\bar x|^2+a\gamma|\bar v|^2,
\end{equation}
where the equality is precisely the reason for choosing $a=m/M_\lambda^2$. Hence
\begin{align}
 \bar z^\top(PA+A^\top P)\bar z
 &\leq-\frac m\gamma|\bar x|^2
       -(a\gamma-2/\gamma)|\bar v|^2\notag\\
 &\leq-\frac m\gamma|\bar x|^2-\frac{a\gamma}{2}|\bar v|^2\notag\\
 &\leq-\frac m\gamma\bigl(|\bar x|^2+a|\bar v|^2\bigr)
 \leq-\frac{2m}{3\gamma}\norm{\bar z}_P^2.\label{eq:leading-final}
\end{align}
Here $a\gamma^2\geq4$ gives the second inequality, $\gamma^2\geq4\cM\geq2m$ gives the third, and \eqref{eq:equivalence} gives the last. 
% \subsection{Absorption and iteration}
Combining \eqref{eq:expansion}, \eqref{eq:remainder-bound}, and \eqref{eq:leading-final} yields
\[
 \norm{\bar z^+}_P^2\leq
 \left(1-\frac{2m\lambda}{3\gamma}+14\lambda^2\gamma^2\right)
 \norm{\bar z}_P^2.
\]
Since $\lambda\leq m/(42\gamma^3)$,
\[
 14\lambda^2\gamma^2\leq\frac{m\lambda}{3\gamma}.
\]
Thus $\norm{\bar z^+}_P^2\leq(1-\lambda\kappa)\norm{\bar z}_P^2$. Iterating completes the proof.
\begin{corollary}
    We have \[W_{a,b} \left(\mathcal{L}\left( z_{n+1}\right),\mathcal{L}\left(\tilde{z}_{n+1}\right)\right)\leq (1-\lambda \kappa)^\frac{1}{2} W_{a,b} \left(\mathcal{L}\left( z_{n}\right),\mathcal{L}\left(\tilde{z}_n \right)\right)\]
\end{corollary}
\subsection{Proof of Log-Sobolev inequality for the OBABO scheme}
Let $S(x,v)=||(x,v)||_{a,b}.$
We denote $P_{a,b}$ the corresponding Markov semigroup induced the transition operator $S(:)\rightarrow S(OBABO(:))$
\begin{lemma}
    We have for all bounded Lispchitz functions $f$, 
    \[|\nabla P_{a,b} f(z)|\leq r P_{a,b} |\nabla f(z)|\]
    where $r=\sqrt{1-\lambda \kappa}$.
\end{lemma}

By the one-step a.s contraction of the OBABO scheme, denoting $r=\sqrt{1-\lambda \kappa}$ one deduces that starting from different intiliazations $z_0=(x_0,v_0)$ and $z'_0=(x'_0,v'_0)$ we have
\[|Sz_1-Sz'_1|\leq r |S z_0-Sz'_0|\] a.s which implies the $W_\infty$ contraction of the Markov operator $P_{a,b}$. By Theorem 2.2 in \cite{kuwada2010duality} one obtains the result.
\begin{lemma}
    Let $z\in \mathbb{R}^d.$ Then, $P_{a,b}(z)$ satisfies a Log-Sobolev inequality with constant  \[C_1\leq 2||S||_{op}^2\left(1+\eta^2(1+\lambda^2 M_\lambda / 2)^2\right)(1-\eta^2)\]
\end{lemma}
\begin{proof}
  \[  \begin{aligned}
&\text { For } x, v, g, g^{\prime} \in \mathbb{R}^d \text {, denote }\\
&\begin{aligned}
& \Theta_1\left(x, v, g, g^{\prime}\right)=x+\lambda\left(\eta v+\sqrt{1-\eta^2} g\right)-\frac{\lambda^2}{2} \hl(x) \\
& \Theta_2\left(x, v, g, g^{\prime}\right)=\eta^2 v-\frac{\lambda \eta}{2}\left(\hl(x)+\hl\left(\Theta_1\left(x, v, g, g^{\prime}\right)\right)\right)+\sqrt{1-\eta^2}\left(\eta g+g^{\prime}\right)
\end{aligned}
\end{aligned}\]
$\Theta=\left(\Theta_1, \Theta_2\right)$. The transition of the OBABO chain is given by $\Theta\left(x_1, v_1\right)= \Theta\left(x_0, v_0, G, G^{\prime}\right)$ with independent $G, G^{\prime} \sim \mathcal{N}\left(0, I_d\right)$.
For fixed $z=(x,v)$ one notices that the OBABO chain is Lipschitz map of $N(0,I_{2d})$ since for G=(g,g')
\[|\Theta_1(z,g_1,g_1')-\Theta_1(z,g_2,g_2')|\leq \lambda \sqrt{1-\eta^2}|g_1-g_2|\]
and \[\begin{aligned}
    |\Theta_2(z,g_1,g_1')-\Theta_2(z,g_2,g_2')|&\leq \frac{\lambda \eta}{2}M_\lambda |\Theta_1(z,g_1,g_1')-\Theta_1(z,g_2,g_2')| \\&+\sqrt{1-\eta^2}\left(\eta |g_1-g_2|+|g_1'-g_2'|\right)
    \\&\leq \sqrt{\left(1+\eta^2(1+\lambda^2 M_\lambda / 2)^2\right)}\sqrt{1-\eta^2}|G_1-G_2|
\end{aligned}\]
which shows that $\Theta$ is Lispchtiz with constant $2 \sqrt{1-\eta^2}$
Since $P_{a,b}(z)=S(S^{-1}z,\Theta(G))$ then it is a Lipschitz function of $G$ with constant $|S|_{Lip}\sqrt{\left(1+\eta^2(1+\lambda^2 M_\lambda / 2)^2\right)}  \sqrt{1-\eta^2}.$
Since the distribution of $G$ satisfies LSI with constant $2$ one obtains the result.
\end{proof}
\begin{lemma}
Assume that $\mu$ satisfies an LSI with constant $C_\mu.$ Then, $\mu P_{a,b}$ also satisfies an LSI
with constant $r^2 C_\mu +C_1$ where $C_1$ is given in the previous Lemma.
\end{lemma}
\begin{proof}
    See proof of Theorem 23, point 1, in \cite{monmarche2021high}.
\end{proof}
% \begin{theorem}
%     Let $z_0=(x_0,v_0)$ be the initial conditions such that $\mathcal{L}(S(z_0))$ satisfies and LSI with constant $C_{LSI,0}$ 
%   Then $\mathcal{L}(S(x_n,v_n))$ satisfies an LSI with \[C_{LSI,n}\leq r^{2n} C_{LSI,0} + \frac{1-r^{2n}}{1-r^{2}}C_1\]
%   where $r:=\sqrt{1-\lambda \kappa}.$
% \end{theorem}
\begin{proof}[Proof of Theorem \ref{th-OBABO-LSI}]
    
    Using the statement of the previous Lemma inductively,
    \[C_{LSI}(S(x_n,v_n))\leq r^{2n} C_{LSI}(S(x_0,v_0)) + \frac{1-r^{2n}}{1-r^{2}}C_1\]
    Using that \[C_{LSI}(\mathcal{L}((x_n,v_n)))\leq ||S^{-1}||^2_{op}C_{LSI}(\mathcal{L}\left(S(x_n,v_n))\right) \]
    and \[ C_{LSI}(\mathcal{L}(S(x_0,v_0)))\leq ||S||^2_{op}C_{LSI}(\mathcal{L}\left((x_0,v_0))\right)\]
    yields the result.
\end{proof}

\section{Convergence of the Monotonic polygonal Stochastic exponential scheme}
\subsection{Convergence in Wasserstein distance}

\begin{definition}\label{cont-timeint-X}
        Let $n\in \mathbb{N}$. Let a random variable $X$.
        We define for $t\in [0,\lambda]$
        \[\z_t=(\x_t,\V_t)\] the one step movement of an algorithm started at a random variable $X$.
\end{definition}
\begin{lemma}\label{Inv mom}
Let $p \geq 2$ and let $Y$ be a random variable such that $\mathcal{L}(Y)=\pi.$ Then
\[ \E|Y|^p\leq   2^{p-1} ((\frac{d}{ m})^{\frac{p}{2}} (1+p/d)^{\frac{p}{2}-1} +\left[\frac{2}{m}\left(\frac{|h(0)|^2}{m}+d\right)\right]^{p / 2}:= C_{\pi,p}. \]
\end{lemma}
    \begin{proof}
Since $e^{- u(x)}=e^{-( u- \frac{m}{2}|x|^2)} e^{- \frac{m}{2}|x|^2}  $ and since the function $u-\frac{m}{2}|x|^2$ is convex (due to strong convexity of $u$) the assumptions of Theorem \ref{harge} are valid for $X$ which has distribution $\mathcal{N}(0,\frac{1}{ m}I_d)$, and for $Y$ with density $\pi$. Since the function $g: x\rightarrow |x|^p$ is convex, applying the results of Theorem \ref{harge} leads to
\begin{equation}\label{pol mom}
    \E |Y-\E Y|^p\leq \E |X|^p=(\frac{d}{ m})^{\frac{p}{2}} \frac{\Gamma((d+p)/2)}{\Gamma(d/2)(d/2)\frac{p}{2}}\leq (\frac{d}{ m})^{\frac{p}{2}} (1+p/d)^{\frac{p}{2}-1}.
\end{equation}
Combining \eqref{pol mom} and the result of Lemma \ref{inv-measure bound} yields
\begin{equation}
    \E|Y|^p\leq 2^{p-1}\left( \E|Y-\E Y|^p + (\E |Y|)^p\right)\leq 2^{p-1} ((\frac{d}{ m})^{\frac{p}{2}} (1+p/d)^{\frac{p}{2}-1} +\left[\frac{2}{m}\left(\frac{|h(0)|^2}{m}+d\right)\right]^{p / 2}
\end{equation}
which completes the proof.
\end{proof}
Since $\Pi$ is the product of $\pi$ and a standard Gaussian the following result follows easily.
\begin{lemma}
    Let $(X,V)$ such that $\mathcal{L}(X,V)=\Pi$. Let $p>0.$ Then, 
    \[\E|X|^{2p}\leq C_{\Pi,2p}\]
    where $C_{\Pi,2p}\leq \mathcal{O}(d^p)$
\end{lemma}
\begin{proof}
    \[\E_{\Pi}|(X,V)|^{2p}\leq  2^{p-1}\left( \E |X|^{2p} + \E |V|^{2p}\right)\]
    Since $X$ follows $\pi$ and $V$ follows a Gaussian one obtains the result, by using the previous Lemma and the moments of the Gaussian.
\end{proof}

% \begin{lemma}\label{lemma-zt mom}
%     Let $X$ such that $\mathcal{L}(X)=\Pi$. Then, for $\z_t$ defined in Definition \ref{cont-timeint-X} there holds
%     \[\sup_{t\in [0,\lambda]} \E |\z_t|^{2p}\leq C_{z,p}\]
% \end{lemma}
% \begin{lemma}\label{one-step}
%     Let  $X$ such that $\mathcal{L}(X)=\Pi$.
%     Then,
%    \[ \E|\x_t-\x_{0}|^{2} \leq C_{os,2} \lambda^2   \]
%    and \[\E |\x_t-\x_{0}|^4 \leq C_{os,4}\lambda ^4 \]
%    where $C_{os,p}\leq \mathcal{O}(d^{p})$
% \end{lemma}
% \begin{proof}
%     \[\E |\V_t|^2\leq 3 (1-e^{-\lambda (t-\lambda n)})^2 \E |\V_{0}|^2 + 6\lambda\gamma d + 3 \lambda |\hl(x_{0}|^2 \leq  
%  C( \E |V_{0}|^2 + \E |x_{0}|^2 + \lambda \gamma d+4 )\]
%  where the last step is obtained by using the fact that the drift coefficient has linear growth.
%  Thus, 
% \[\E |\x_t-\x_{0}|^2 \leq \E \left(\int_{0}^t \V_s ds\right)^2 \leq \lambda \int \E |\V_s|^2 ds \] which completes the first proof.
% With the same arguments,
% \[\E |V_t|^4\leq C'(\E |\V_{0}|^4+ \E |\x_{0}|^4 + \lambda^2 \gamma^2 d^2+ 1)\] 
% and since 
% \[\E |\x_t-\x_{0}|^4 \leq \E \left(\int_{0}^t \V_s ds\right)^4 \leq \lambda^3 \int \E |\V_s|^4 ds \] the result immediately follows.
% \end{proof}
% \begin{lemma}
%     Let $\z_t=(\x_t,\V_t)$ as in Lemma \ref{lemma-zt mom}.
%     Then,
%     \[\E |\hl(\x_{0})-h(\x_{0})|^2\leq C_{tam}\lambda^2\]
% \end{lemma}
\begin{lemma}\label{one-step-approx}
    Let $X$ such that $\mathcal{L}(X)=\Pi$. Let the Langevin SDE $(p_t,Q_t)$ as in \eqref{eq-underlangtwo} started at time $0$ at $X$ and $\z_t=(\x_t,\V_t)$ the one-step movement of the algorithm as in Definition \ref{cont-timeint-X}. 
    We have, 
    \[W_2\left(\mathcal{L}(\z_t),\Pi\right)\leq C\lambda^2\]
    As a result,
    \[W_{M} \left(\mathcal{L}(\z_t),\Pi\right)\leq \sqrt{3}C \lambda^2\]
\end{lemma}
\begin{proof}
   Since $\Pi$ is the invariant measure of the kinetic Langevin SDE \eqref{eq-underlangtwo}, $\mathcal{L}((p_t,Q_t))=\Pi$ so it suffices to bound $\E |\z_t-(p_t,Q_t)|^2$.
  Since the process have the same initial conditions, one notices (using Jensen's inequality).
   \begin{equation}\label{eq-xp}
   \begin{aligned}
         \E |\x_t-p_t|^2\leq \E |\int_{0}^t \V_s-Q_sds |^2 &\leq (t-0) \int\E |\V_s-Q_s|^2 ds\\&\leq \lambda \int_{0}^t \E |\V_s-Q_s|^2 ds
   \end{aligned}
   \end{equation}
   With the same argument,
   \begin{equation}
       \begin{aligned}
     \E |\V_t-Q_t|^2& \leq       \E |\int_{0}^t e^{-\gamma (t-s)}(h(p_s)-\hl(\x_{0}))ds |^2 ds
     \\&\leq \lambda \int_{0}^t e^{-2\gamma(t-s)}\E |h(p_s)-\hl(\x_{0})|^2 ds
   \\&\leq \lambda \int_{0}^t \E |h(p_s)-\hl(\x_{0})|^2 ds
     \\&\leq 2 \lambda \int_{0}^t \E |h(p_s)-h(\x_{0})|^2 ds
     \\& + 2\lambda \int_{0}^t \E |h(\x_{0})-h_\lambda(\x_{0})|^2 ds
     \\&\leq 2 \lambda \int_{0}^\lambda \E (1+|p_s|+\x_{0})^{2l} |p_s-\x_{0}|^2
           \\& + 2\lambda \int_{0}^\lambda \E |h(\x_{0})-h_\lambda(\x_{0})|^2 ds.
           \\&\leq 2 \lambda \int_{0}^\lambda \sqrt{\E (1+|p_s|+\x_{0})^{4l}} \sqrt{\E \left|\int_0^s V_u du \right |^4 }
          \\& 2\lambda \int_{0}^\lambda \E |h(\x_{0})-h_\lambda(\x_{0})|^2 ds.
          \\&+  2\lambda \int_{0}^\lambda \E |h(\x_{0})-h_\lambda(\x_{0})|^2 ds.
          \\&\leq 2 \lambda \int_{0}^\lambda \sqrt{\E (1+|p_s|+\x_{0})^{4l}} \sqrt{\lambda^3 \int_0^\lambda \E  |V_u|^4 du  }
          \\&+  2\lambda \int_{0}^\lambda \E |h(\x_{0})-h_\lambda(\x_{0})|^2 ds
    \end{aligned}
     \end{equation}
     Because $\mathcal{L}(\x_0)=\pi,$ and also $\mathcal{L}(p_s)=\pi$ and $\mathcal{L}(V_s)=N(0,I_d)$ (since $\Pi$ is invariant for the Langevin SDE, the first term can be bounded by the relative moment bounds of the invariant measure, so it can be bounded by $C_1 \lambda^4$ where $C_1\leq d^{2l+2}$, while the second term is $\mathcal{O}(\lambda^6)$ as can be seen by the taming error given in \eqref{eq-tamingerror} which gives a $\mathcal{O}(\lambda^4)$ for the velocity variables. Applying ths bound in \eqref{eq-xp} yields the result.
\end{proof}
\begin{theorem}
    We have \[W_{M}\left( \mathcal{L}(\theta_n,V_n),\Pi \right)\leq  (1-f(\lambda))^{n/2}   W_{M}\left( \mathcal{L}(\theta_{0},V_0),\Pi \right) + C\lambda^2 /f(\lambda) \]
    Using the connection between $W_{M}$ and $W_2$ one obtains
    \[W_{2}\left( \mathcal{L}(\theta_n,V_n),\Pi \right)\leq  2 \gamma(1-f(\lambda))^{n/2}   W_{2}\left( \mathcal{L}(\theta_{0},V_0),\Pi \right) + 3C\lambda^2 \gamma /f(\lambda)\]
\end{theorem}
\begin{proof}
Let $\hat{P}$ be the Markov operator given by a single step of the algorithm. Let $\mu_0=\mathcal
L(\theta_0,V_0)$
    \begin{equation}
        \begin{aligned}
            &W_M\left( \mu_0\hat{P}^{n+1},\Pi \right)\\&\leq   W_M\left( (\mu_0\hat{P}^n)\hat{P},\Pi\hat{P}\right) +
            W_M\left(\Pi \hat{P},\Pi\right)
            \\&\leq (1-f(\lambda))^\frac{1}{2} W_{M}(\mu_0\hat{P}^n,\Pi) + \sqrt{\frac{3}{2}}C  \lambda^2 
        \end{aligned}
    \end{equation}
    By induction, one deduces that 
   \[ W_{M}\left( \mu_0\hat{P}^{n+1},\Pi \right)\leq  (1-f(\lambda))^{n/2}   W_{M}\left(\mu_0 ,\Pi \right) + 2C \lambda^2 /f(\lambda)\]
   As a result,
   \[W_2\left( \mathcal{L}(\theta_{n+1},V_{n+1}),\Pi \right)\leq \gamma (1-f(\lambda))^{n/2}   W_{2}\left( \mathcal{L}(\theta_{0},V_0),\Pi \right) +  2C \gamma \lambda^2 /f(\lambda) \ \]
\end{proof}
\subsection{Convergence in total variation}
Recall that we have shown in the previous sections that the algorithm can be formulated as 
\[(\phi_{n+1},u_{n+1})=F(\phi_n,u_n)+ \bar{\Sigma} \otimes I_d \]
where $(\phi_n,u_n)=(x_n+\frac{2}{\gamma} v_n,v_n)=M(x_n,v_n)$
\begin{theorem}[\cite{altschuler2024faster}, Theorem 3.1]
 (Shifted Rényi divergence analysis). Consider two Markov chains $\left\{\mu_n\right\}_{n  \geq 0}$ and $\left\{\mu_n^{\prime}\right\}_{n  \geq 0}$ with possibly different initialization, but with the same update transitions.
In order to proceed we will utilize a recently proved contraction theorem for Renyi divergence.
$$
\begin{aligned}
& \mu_{n+1}=\left(\mu_n P_n\right) * \mathcal{N}\left(0, \sigma^2 I_d\right) \\
& \mu_{n+1}^{\prime}=\left(\mu_n^{\prime} P_n\right) * \mathcal{N}\left(0, \sigma^2 I_d\right)
\end{aligned}
$$
where $P_n$ is a Markov transition kernel that is $c$-Lipschitz in the $W_{\psi_2}$ metric. Then for any Rényi order $q  \geq 1$,

$$
\mathcal{R}_q\left(\mu_N \| \mu_N^{\prime}\right) \leq c^{2 N} \frac{q W_{\psi_2}^2\left(\mu_0, \mu_0^{\prime}\right)}{2 \sigma^2},
$$

as long as $N  \geq \log _{1 / c}\left(\frac{\sqrt{q(q-1)} W_{\psi_2}\left(\mu_0, \mu_0^{\prime}\right)}{\sigma \sqrt{2}}\right)$.
\end{theorem}
\begin{corollary}\label{cor-renyi}
    Let $\nu_n^{\nu_0}=\mathcal{L}(x_n^{\nu_0},v_n^{\nu_0}).$ For $n\geq \log_{c^{-1}} \left( \frac{W_{\psi_2}(\mathcal{L}(\theta_0,V_0),\Pi)}{\sigma}\right)$ 
    where $c=\sqrt{1- \frac{\lambda m}{2\gamma}}$ and $\sigma=\frac{\gamma \lambda^3}{6}(1-c'\lambda \gamma),$ we have 
    \[R_2(\nu_n^{(x_0,v_0)},\nu_n^{\Pi})\leq  e^{-2c_0\lambda \gamma^{-1}n} W^2_{\psi_2}(\mathcal{L}(x_0,v_0),\Pi)\frac{C}{\lambda^3 \gamma} \]
\end{corollary}
\begin{proof}
    Recall that $||F||_{lip}= \sqrt{1-\frac{m\lambda}{2\gamma}}\leq e^{-c_0\lambda \gamma^{-1}}.$
   \\ We write $\bar{l}=\eig_{min} \bar{\Sigma}=\frac{\gamma \lambda^3}{6}(1-O(\gamma h)) .$
    Since $M$ is an invertible function we have that
  for $\mu_n=L(x_n,v_n)$ and $\nu_n=M_{\#} \mu_n$ since $M$ is invertible by the data processing inequality,
  \[R_2(\nu_n,\nu'_n)=R_2(M_{\#} \mu_n,M_{\#} \mu'_n)\leq R_2(\mu_n,\mu'_n)\] and 
   \[R_2(\mu_n,\mu'_n)=R_2(M^{-1}_{\#} \nu_n,M^{-1}_{\#} \nu'_n)\leq R_2(\nu_n,\nu'_n)\]
   yields \[R_2(\nu_n,\nu'_n)=R_2(\mu_n,\mu'_n).\]
   Setting $Q$ the operator such that $\nu Q=F_{\#}\nu*N(\bar{\Sigma}-\bar{l}I_d)$. Thus \[\nu_{n+1}=(\nu_n Q)*N(0,\bar{l}I_d).\]
   Since $F$ is Lipschitz, $Q$ is also Lipchitz in the $W_{\psi_2}$ metric with the same constant so 
   one can apply the previous theorem to obtain
   \[R_2(\mu_n,\mu'_n)=R_2(\nu_n,\nu'_n)\leq C e^{-c_0\lambda \gamma^{-1}n} W_{\psi_2}(\mu_0,\mu_0')\left(\frac{1}{\gamma \lambda^3}\right)\]
\end{proof}
\begin{proposition}\label{prop-Kl est}
    Let $L(\hat{x}_{\lambda_n}^\Pi,\hat{v}_{\lambda_n}^\Pi)$ the law of the $n-th$ iterate of the stochastic tamed exponential scheme started at $\Pi$. We have 
    \[TV(\mathcal{L}(\hat{x}_{\lambda_n}^\Pi,\hat{v}_{\lambda_n}^\Pi),\Pi)\leq \mathcal{O}\left(\sqrt{\lambda^3\gamma^{-1} n}\right)\]
\end{proposition}
We want to bound $TV(\mathcal{L}(\hat{x}_{\lambda_n}^\Pi),\mathcal{L}(x_{\lambda_n}^\Pi)).$
By Pinsker inequality we are going to bound. \[TV(\mathcal{L}(\hat{x}_{\lambda_n}^\Pi,\hat{v}_{\lambda_n}^\Pi),\mathcal{L}(x_{\lambda_n}^\Pi,v_{\lambda n}^\Pi))\leq \sqrt{2KL(\mathcal{L}(x_{\lambda_n},v_{\lambda n}))||\mathcal{L}(\hat{x}_{\lambda_n},\hat{v}_{\lambda n}))}\]
By Girsanov theorem, (for which Novikov condition can be justified by standard localization arguments and the lower semi-continuity of $KL$)
\[KL(\mathcal{L}(x_{\lambda_n},v_{\lambda n}))||\mathcal{L}(\hat{x}_{\lambda_n},\hat{v}_{\lambda n}))\leq \sum_{k=0}^n\E_P \frac{1}{2\gamma} \int_{k\lambda}^{(k+1)\lambda}|h_\lambda(x_{\lambda k})-h(x_s)|^2 ds\]
where $P$ is the path measure such that $\mathcal{L}(x_{k\lambda})=\mathcal{L}(x_s)=\pi.$
Since by the approximation of the tamed scheme and the moments of the invariant measure one obtains \[\E_P |h_\lambda(x_{k\lambda})-h(x_{k\lambda})|^2\leq C_{\pi,2l} \lambda^2,\]
and \[\begin{aligned}
    \E_P|h(x_s)- h(x_{k\lambda})|^2&\leq \E_P (1+|x_s|+|x_{k\lambda}|)^{2l}|\int_{k\lambda}^s V_udu|^2\\&\leq \sqrt{\E_P (1+|x_s|+|x_{k\lambda}|)^{4l}}\lambda^2 \sqrt{\sup_{k\lambda\leq u \leq (k+1)\lambda} \E_P |V_u}|^4 
\end{aligned}\]
and $\mathcal{L
}(V_u)=N(0,I_d)$ one obtains that \[KL(\mathcal{L}(x_{\lambda_n}^\pi)||\mathcal{L}(\hat{x}_{\lambda_n}^\pi))\leq C \gamma^{-1}\lambda^3 n\]
\begin{proof}[Proof of convergence in total variation]
    Noticing that \[TV\left(\mathcal{L}(x_n^{(x_0,v_0)},v_n^{(x_0,v_0)}),\Pi\right)\leq TV\left(\mathcal{L}(x_n^{(x_0,v_0)},v_n^{(x_0,v_0)}),\mathcal{L}(x_n^{\Pi},v_n^{\Pi})\right) + TV(\mathcal{L}(x_n^{\Pi},v_n^{\Pi}),\Pi). \]
    Using the fact that Pinsker's inequality implies $TV\leq \sqrt{2 KL}\leq \sqrt{2R_2}$ the result follows from Proposition \ref{prop-Kl est} and Corollary \ref{cor-renyi}.
\end{proof}

\section{Convergence of the OBABO scheme in Wasserstein distance}

% \begin{corollary}
% For all $p \geq 1$ and all $\nu,\mu\in\mathcal P_p(\R^{2d})$, 
%  \begin{eqnarray}\label{Eq:W2contract}
%  \mathcal W_{a,b}^2\po \nu P,\mu P\pf & \leq  & (1-\lambda\kappa) \mathcal W_{a,b}^2\po \nu ,\mu \pf\,,
%  \end{eqnarray}
%  and for all $n\in \N$,
%  \[\W_2^2\po \nu P^n,\mu P^n\pf \ \leq 3 M_\lambda  (1-\lambda \kappa)^{n} \W_2^2(\nu,\mu)\,.\]
%  Moreover, $P$ admits a unique invariant probability distribution $\pi_\lambda$, and $\pi_\lambda\in  \mathcal P_p(\R^{2d})$ for all $p \geq 1$. 
% \end{corollary}
% \begin{proof}
%     \eqref{Eq:W2contract} is a direct consequence of the previous result.\\
%     Using the connection of the $|\cdot |_{a,b}$ and the Euclidean distance distance in \eqref{eq-equivdist}  one easily obtains the result for the Wasserstein distance. The existence and uniqueness of the invariant measure is a consequence of Banach fixed point theorem.
% \end{proof}

    Let $P_O$ the Markov kernel associated with $O$-step in \eqref{eq-decomp-OBABO}. The Markov kernel $P$ associated with the whole chain can be written as
    \begin{equation}
        P=P_O P_V P_O
    \end{equation}
    \begin{lemma}\label{Po-inv}
    $\Pi$ is invariant under $P_O$ i.e 
    \[\Pi P_0 =\Pi\]
    in distribution.
    \end{lemma}
    \begin{proof}
        Writing $\Pi=(\Pi_1,\Pi_2)$ the first part follows $\pi$ while the second follows $N(0,I_d).$ The $O$-step in \eqref{eq-decomp-OBABO}, only acts on $\Pi_2$. Since the Gaussian random variable $G$ added is independent from $\Pi_2$, the resulting distribution will still be a Gaussian with mean $0$ and covariance $(\tilde{\eta}^2+(1-\tilde{\eta}^2))I_d=I_d$
    \end{proof}
    \begin{lemma}\label{Po-nonincr}
        For any measures $\mu$, $\nu$ $\in P_p(\mathbb{R}^d)$ there holds
        \[W_2(\mu P_O, \nu P_O)\leq W_2(\mu,\nu) \]
    \end{lemma}
    \begin{proof}
      Let $\left(x_i, v_i\right) \in \mathbb{R}^{2 d}, i=1,2$, and $G \sim \mathcal{N}\left(0, I_d\right)$. Set $x_i^{\prime}=x_i$ and $v_i^{\prime}=\eta v_i+\sqrt{1-\eta^2} G$ for $i=1,2$. Then $\left(x_i^{\prime}, v_i^{\prime}\right) \sim \lambda_{\left(x_i, v_i\right)} P_O$ and

$$
\left|x_1^{\prime}-x_2^{\prime}\right|^2+\left|v_1^{\prime}-v_2^{\prime}\right|^2 \leq\left|x_1-x_2\right|^2+\left|v_1-v_2\right|^2
$$

almost surely. \\Considering $\nu_i \in \mathcal{P}_p\left(\mathbb{R}^d\right)$ for $i=1,2$, sampling $\left(x_i, v_i\right) \sim \nu_i$, using the previous coupling. Taking the power $p / 2$ and the expectation yields
\[\E |(x_1^{\prime},v_1^{\prime})-(x_2^{\prime},v_2^{\prime})|^{2}\leq W_2^2(\nu_1,\nu_2)\]
which leads to the result.
    \end{proof}
   
    \newcommand{\pd}{\Pi_\lambda}

    \newcommand{\Wp}{W_{a,b}}
    \begin{proposition}\label{prop-base-OBAB}
    Let $\Pi_\lambda$ be the invariant measure of the tamed $OBABO$ algorithm. Then,
       \[W_{a,b}\left(\Pi_\lambda, \Pi\right) \leq \frac{2}{\lambda \kappa} W_{a,b}(\Pi P, \Pi)\leq \sqrt{\frac{3}{2}}\frac{2}{\lambda \kappa} W_2\left(\Pi P_V, \Pi\right) \] 
    \end{proposition}
        \begin{proof}
            Writing \[\begin{aligned}
                \Wp (\pd,\Pi)&=\Wp(\pd P,\Pi)\leq \Wp (\pd P, \Pi P)+ \Wp(\Pi P,\Pi)
                \\&\leq \sqrt{(1-\lambda \kappa)} \Wp(\pd,\Pi)+ \Wp(\Pi P,\Pi)
                \\&\leq (1- \frac{\lambda \kappa}{2})\Wp(\pd,\Pi)+ \Wp(\Pi P,\Pi)
            \end{aligned}\]
        leads to 
        \begin{equation}\label{eq-bound pd-p}
            \Wp (\pd,\Pi)\leq \frac{2}{\lambda \kappa}\Wp(\Pi P,\Pi)\leq
             \sqrt{\frac{3}{2}}\frac{2}{\kappa \lambda} W_2(\Pi P,\Pi)
        \end{equation}
        where the last step is obtained by \eqref{eq-equivdist}.
        In addition, noting that since $\Pi$ is invariant under the $O$-step due to Lemma \ref{Po-inv}, then 
\[W_2(\Pi P,\Pi)=W_2(\Pi P_O P_V P_O, \Pi)=W_2( \Pi P_V P_0, \Pi P_0)\]
and applying the result of Lemma \ref{Po-nonincr} one obtains
\begin{equation}\label{eq-PV}
    W_2(\Pi P,\Pi)\leq W_2( \Pi P_V,\Pi).
\end{equation}
        Combining \eqref{eq-PV} and \eqref{eq-bound pd-p} completes the proof.
        \end{proof}
\subsection{Comparing the  BAB steps to a Hamiltonian system}
In order to complete our proof of the bound of the Wasserstein distance of $\pd$ and $\Pi$ one need to control $W_2(\Pi,\Pi P_V).$
We introduce the transition operator $Q_t$ of the Hamiltonian dynamics: given $(x, v) \in \mathbb{R}^{2 d}$ and considering $\left(\hat{x}_t, \hat{v}_t\right)_{t  \geq 0}$ the solution of

\begin{equation}\label{eq-hamiltonian}
    \hat{x}_t^{\prime}=\hat{v}_t, \quad \hat{v}_t^{\prime}=-\nabla u \left(\hat{x}_t\right), \quad \hat{x}_0=x, \quad \hat{v}_0=v,
\end{equation}

denote $\Phi_{H D}(x, v)=\left(\hat{x}_\lambda, \hat{v}_\lambda\right)$ and $Q_\lambda \varphi(x, v)=\varphi\left(x_\lambda, v_\lambda\right)$.\\

In addition, the Verlet integrator BAB can be presented as
\[\begin{aligned}
&\Phi_V(x, v)=\left(x+\lambda v-\frac{1}{2} \lambda^2 \hl(x), v-\frac{1}{2} \lambda\left(\hl(x)+\hl\left(x+\lambda v-\frac{1}{2} \lambda^2 \hl(x)\right)\right)\right)\\
&\text { so that the steps } \mathrm{BAB} \text { are given as }\left(x_1, v_1^{\prime}\right)=\Phi_V\left(x_0, v_0^{\prime}\right)
\end{aligned}\]
\begin{lemma}
    \[\sup_{t\in [0,\lambda]}|\hat{x}_t|^{p}+|\hat{v}_t|^{p}\leq C(1+ |u(x)|^{p}+|v|^p) \quad \forall p \in \mathbb{N}.\]
    Suppose that the intial condition is the measure $\Pi$ then,
    \[\E |\hat{x}_t|^{p}+\E |\hat{v}_t|^{p}\leq \hat{C}_p \quad \forall t \in [0,\lambda]\] 
\end{lemma}
\begin{proof}
    Since for large values, $u$ is superlinear and the fact that in a Hamiltonian system the energy is preserved then,
   \[ |\hat{x}_t|^{p}+|\hat{v}_t|^{p}\leq C(1+u(\hat{x}_t)+|\hat{v}_t|)^p \leq C(1+ u(x)^p+ |v|^p).\]
   The second claim follows by taking expectations.
\end{proof}
\begin{lemma}
    We have 
    \[W_2(\Pi, \Pi P_V)=W_2(\Pi Q_\lambda, \Pi P_V)\leq \mathcal{O}\left( M_\lambda \lambda^2\right) \]
\end{lemma}
\begin{proof}
Let $(x_t,w_t)$ the solution of 
\[   \begin{aligned}
&x_t^{\prime}=w_t, \quad w_t^{\prime}=-\hl(x), \quad x_0=x, \quad w_0=v\\
&\text { In other words, }\\
&x_t=x+t v-\frac{t^2}{2} \hl(x), \quad w_t=v-t \hl(x),
\end{aligned}\]
where $\mathcal{L}(x,v)=\Pi$
Thus, 
\[\pv(x,v)=(x_\lambda,\tilde{w}_\lambda)=\left(x_\lambda, w_\lambda  + \frac{\lambda}{2}(\hl(x)-\hl(x_\lambda))\right)\]
For all $t\in [0,\lambda]$,
\begin{equation}\label{eq-same}
    \begin{aligned}
        \E |\hat{x}_t-x_t|^2 &\leq  \E |\int_0^t v_s-v ds +\int_0^t\int_0^s \hl(x)duds|^2
        \\&\leq  \lambda^2\int_0^t \int_0^s \E |\nabla u(\hat{x}_u)-\hl(x)|^2duds
        \\&\leq 2\lambda^2 \int_0^t \int_0^s \E |\nabla u(\hat{x}_u)-\hl(\hat{x}_u)|^2duds +2\lambda^2\int_0^t \int_0^s \E |\hl(x)-\hl(\hat{x}_u)|^2 duds
        \\&\leq C \lambda^4 + \lambda^2 M_\lambda^2 \int_0^t \int_0^s \E |x-\hat{x}_u|^2duds
        \\&\leq  C \lambda^6 + \lambda^2 M^2_\lambda \int_0^t \int_0^s (\int_0^u \E |\hat{v}_z|dz)^2duds
        \\&\leq C \lambda^6 + \lambda^2 M^2_\lambda \int_0^t \int_0^s u\int_0^u \E |\hat{v}_z|^2 dz duds
        \\&\leq C\lambda^6 + C' \lambda^6 M_\lambda^2
    \end{aligned}
\end{equation}
In addition, with the same arguments as before
\begin{equation} \label{eq- w-hatv}
\begin{aligned}
    \E |w_t-\hat{v}_t|^2\leq \lambda \int_0^t \E |\nabla u(\hat{x}_s)-\hl(x)|^2 ds \leq \hat{C}\lambda^4 M_\lambda^2
\end{aligned} \end{equation}
Bringing all together,
\[\begin{aligned}
    &\E |\hat{v}_\lambda -\tilde{w}_\lambda|^2\leq 2 \E |\hat{v}_\lambda -w_\lambda|^2 +2 \frac{\lambda^2}{4}\E |\hl(x_\lambda)-\hl(x)|^2\\&\leq 2\hat{C}\lambda^4 M_\lambda^2 + \frac{\lambda^2}{2} M_\lambda^2 \E |x_\lambda-x|^2
    \\&\leq 2\hat{C}\lambda^4 M_\lambda^2 + \lambda^2 M_\lambda^2 (\lambda^2 |v|^2 + \frac{\lambda^4}{4}|\hl(x)|^2)
    \\&\leq \mathcal{O}(\lambda^4 M_\lambda^2)
\end{aligned} \]
\end{proof}
\begin{proof}[Proof of Theorem \ref{conv-Obabo}]
Combining the previous Lemma and \ref{prop-base-OBAB} one deduces that 
\[W_{a,b}(\Pi_\lambda, \Pi)\leq C M_\lambda^{\frac{3}{2}}  \lambda \]
As a result, denoting 
\[c_{-}(a)=\sqrt{\frac{2}{\min\{a,1\}}}= \sqrt{\frac{2}{a}}\]
and \[c_{+}(a)= \sqrt{\frac{3}{2}\max\{a,1\}}=\sqrt{\frac{3}{2}}\]
\[\begin{aligned}
    W_{a,b}(\mu_0P^n,\Pi)&\leq W_{a,b}(\mu_0P^n,\Pi_\lambda)+W_{a,b}(\Pi,\Pi_\lambda)
    \\&\leq  c_{+}(a) (1-\lambda \kappa)^\frac{n}{2} W_2(\mu_0,\Pi_\lambda) + M_\lambda^{\frac{3}{2}}  \lambda
    \\&\leq c_{+}(a)(1-\lambda \kappa)^\frac{n}{2} \left(W_2(\Pi,\Pi_\lambda)+W_2(\Pi,\mu_0)\right) +M_\lambda^{\frac{3}{2}}  \lambda
   \\&\leq  c_{+}(a)(1-\lambda c'\gamma^{-1})^\frac{n}{2}\left(M_\lambda^{\frac{3}{2}}  \lambda + W_2(\Pi,\mu_0)\right) +M_\lambda^{\frac{3}{2}}  \lambda.
\end{aligned}\]
Using $W_2\leq c_{-}(a)W_{a,b}$ completes the proof.
\end{proof}
\subsection{Improved Wasserstein convergence under increased Hessian smoothness}
We will further assume that \[||\nabla^2 u(x)-\nabla^2u(y)||\leq L' (1+|x|+|y|)^{l'}|x-y|. \]
We shall use the second order property of the OBABO integrator to obtain an increased rate.
\begin{lemma}
    We have \[W_2(\Pi,\Pi P_V)\leq W_2(\Pi Q_\lambda,\Pi P_V)\leq \mathcal{O}(M_\lambda \lambda^3)\]
\end{lemma}
\begin{proof}
The proof essentially changes at the bound of the difference in the velocity variables \[\E |\hat{v}_\lambda-\tilde{w}_\lambda|^2\]
Noticing that \begin{equation}\label{eq-long hatv-w}
\begin{aligned}
\hat{v}_\lambda-\tilde{w_\lambda}&=(\hat{v}_\lambda -v) +\frac{\lambda}{2} (\hl(x)-\hl(x_\lambda))\\&=-\int_0^\lambda \nabla u(\hat{x}_s)ds+\frac{\lambda}{2} (\hl(x)+\hl(x_\lambda))
\\&= -\int_0^\lambda \nabla u(\hat{x}_s)ds + \frac{\lambda}{2}( \nabla u(x)+\nabla u(\hat{x}_\lambda)) + \frac{\lambda}{2} (\hl(x)-\nabla u(x)) \\&-\frac{\lambda}{2} (\nabla u(\hat{x}_\lambda)-\hl(\hat{x}_\lambda)) - \frac{\lambda}{2} (\hl(\hat{x}_\lambda)-\hl(x_\lambda)).
\end{aligned} \end{equation}
In order to proceed we will use a function that will help us bring forward the enhanced properties of the Hessian of the potential. 
    Let \[F(s)=\nabla u(\hat{x}_s).\]
    It follows that \[\begin{aligned}
        |F'(s)-F'(0)|&=|\nabla^2u(\hat{x}_s)\hat{v}_s-\nabla^2 u(x) v|\\&\leq |\nabla^2u(x)(\hat{v}_s-v)|+|(\nabla^2u (\hat{x}_s)-\nabla^2u(x)) \hat{v}_s|
        \\&\leq ||\nabla^2 u(x)|||\hat{v}_s-v| +||\nabla^2 u(\hat{x}_s)-\nabla^2 u(x)|||\hat{v}_s| 
        \\&\leq ||\nabla^2 u(x)|||\hat{v}_s-v| + L' (1+|\hat{x}_s| +|x|)^{l'} |\hat{x}_s-x||\hat{v}_s|
        \\&\leq ||\nabla^2 u(x)|| \left |\int_0^s -\nabla u(\hat{x}_u)du\right| +  L' (1+|\hat{x}_s| +|x|)^{l'} \left|\int_0^s \hat{v}_udu\right||\hat{v}_s|
    \end{aligned}\] 
    where the last steps where obtained by the properties of the Hamiltonian system and local Lipschitz continuity of the Hessian. As
    a result, by successive Cauchy-Swartz inequalities since the moments of $\hat{v}_s$, $v$, $x$, $\hat{x}_s$ are controlled one obtains 
    \begin{equation}\label{eq-div F}
        \E |F'(s)-F'(0)|^2\leq C \lambda^2
    \end{equation}
    Going back to \eqref{eq-long hatv-w} by the taming error, the moment bounds and the Lipschitznes of $\hl$ for the final term it is easy to see that by applying \eqref{eq-same}.
    \begin{equation}
        \begin{aligned}
            \E |\hat{v}_\lambda -\tilde{w}_\lambda|^2 &\leq 16 \E \left| -\int_0^\lambda \nabla u(\hat{x}_s)ds+ \frac{\lambda}{2}( \nabla u(x)+\nabla u(\hat{x}_\lambda))\right |^2 + C (\lambda^6 + \lambda^2 M_\lambda^2 \E |\hat{x}_\lambda-x_\lambda|^2) 
            \\&\leq 16\E\left|-\int_0^\lambda F(s)ds + \frac{\lambda}{2} (F(0)+F(\lambda))\right|^2 + C\lambda^6M_\lambda^2
            \\&\leq 16 \E |\int_0^\lambda (s-\lambda/2) F'(s)ds|^2+ C\lambda^6M_\lambda^2
        \end{aligned}
    \end{equation}
Noticing that \[\int_0^\lambda F'(0)(s-\lambda/2)= F'(0) (\int_0^\lambda s -\lambda^2/2)=0\] then,
\begin{equation}\label{eq-finalv-w}
\begin{aligned}
   \E |\hat{v}_\lambda -\tilde{w}_\lambda|^2&\leq 16 \E |\int_0^\lambda (s-\lambda/2) (F'(s)-F'(0))ds|^2 + C\lambda^6 M_\lambda^2\\&\leq 4 \lambda^3 \int_0^\lambda \E |F'(s)-F'(0)|^2 ds+ C\lambda^6 M_\lambda^2\\&\leq  C \lambda^6 (1+M_\lambda^2) 
\end{aligned}
\end{equation}
where the last step was derived by \eqref{eq-div F}.
Combining \eqref{eq-same} and \eqref{eq-finalv-w} yields the result.
\end{proof}
\section{Convergence of the OBABO scheme in total variation}
In order to derive an expression in total variation, we are going to compare our sampler with the unbiased Metropolis adjusted algorithm with drift $\nabla u$, which preserves the invariant measure.
Classically, this is addressed by incorporating a Metropolis accept/reject step into this part of the dynamics, yielding the Hamiltonian Monte Carlo (HMC) algorithm. We now describe this in detail. Define
\[
\Phi_{MV}(x, v) = \Big(x + \lambda v - \tfrac{1}{2} \lambda^2 \nabla U(x),\; v - \tfrac{1}{2} \lambda \big(\nabla U(x) + \nabla U(x + \lambda v - \tfrac{1}{2} \lambda^2 \nabla U(x))\big) \Big),
\]
so that the BAB steps can be written as \((x_1, v_1') = \Phi_{MV}(x_0, v_0')\). Also define the velocity-flip operator 
\[
\Phi_R(x, v) = (x, -v).
\] 
It is straightforward to verify that the Verlet algorithm is time-reversible, meaning
\[
\Phi_R \Phi_{MV} \Phi_R \Phi_{MV}(x, v) = (x, v) \quad \text{for all } x, v \in \mathbb{R}^d.
\] 

Consequently, the deterministic proposal kernel
\[
q((x, v), \mathrm{d}y \mathrm{d}w) = \lambda_{\Phi_R \Phi_{MV}(x, v)}(\mathrm{d}y \mathrm{d}w)
\]
is symmetric. Let \(P_{MH}\) denote the Markov transition operator of the Metropolis-Hastings algorithm with proposal kernel \(q\). 
Let $H(x,v)= u(x) +\frac{1}{2}|v|^2.$ A single transition of the resulting chain proceeds as follows:
\[
\begin{array}{rlrl}
(x', v') & = \Phi_V(x_0, v_0), & & \\
(x_1, v_1) & =
\begin{cases}
(x', -v'), & \text{if } W \le \alpha(x_0, v_0) := \exp\big(-[H(\Phi_{MV}(x_0, v_0)) - H(x_0, v_0)]_+\big), \\
(x_0, v_0), & \text{otherwise,}
\end{cases}
\end{array}
\]
where \(W \sim \mathrm{Uniform}[0,1]\) and we use that \(H(x', v') = H(x', -v')\). By construction, \(P_{MH}\) is probabilistically reversible with respect to \(\Pi\). Moreover, \(\pi\) is invariant under the velocity-flip transformation \(v \mapsto -v\), i.e., 
\[
\Pi P_R = \Pi \quad \text{with} \quad P_R \varphi(x, v) = \varphi(x, -v).
\]
Define
\[
P_{MV} = P_R P_{MH},
\]
which corresponds to a BAB-style Verlet transition as before with drift $\nabla u$ instead of $\hl$, but also incorporates the Metropolis accept/reject step.
The metropolis adjusted algorithm is given as \[P_M=P_O P_{M V} P_O\] where $O$ is the $O$-step in the OBABO chain.

The operator  $P_M$  corresponds to the following transition: \\
\[\begin{aligned}
v_0^{\prime} & =\eta v_0+\sqrt{1-\eta^2} G & & \\
v_{1 / 2} & =v_0^{\prime}-\lambda / 2 \nabla U\left(x_0\right) & & \\
\tilde{x}_1 & =x_0+\lambda v_{1 / 2} & & \\
\tilde{v}_1 & =v_{1 / 2}-\lambda / 2 \nabla U\left(\tilde{x}_1\right) & & \\
\left(x_1, v_1^{\prime}\right) & =\left(\tilde{x}_1, \tilde{v}_1\right) & & \text { if } W \leq \alpha\left(x_0, v_0^{\prime}\right) \\
& =\left(x_0,-v_0^{\prime}\right) & & \text { otherwise } \\
v_1 & =\eta v_1^{\prime}+\sqrt{1-\eta^2} G^{\prime} . & &
\end{aligned}\]
Let us now compare the transition $P_{MV}$ with the transition $P_V$ acting on the same initial condition, which is distributed according to $\Pi$.
\begin{lemma}\label{Lemma-comp with metrop}
    \[||\Pi P_V-\Pi||_{TV} =||\Pi P_V-\Pi P_{MV}||_{TV}\leq \mathcal{O}(\lambda^2).\]
\end{lemma}
\begin{proof}
The first part of the proof is satisfied since $\Pi P_{MV}=\Pi.$
% For the rest of the proof, we consider a uniformly distributed random variable in $[0,1]$, $(X',V')$ are distributed according to $\Pi P_{MV}$
%     \[\left(X^{\prime}, V^{\prime}\right)=(X,-V) \mathbb{1}_{W>\alpha(X, V)}+\Phi_{MV}(X, V) \mathbb{1}_{W \leq \alpha(X, V)}\]
%     while $\Phi_V(X,V)$ is distributed according to $\Pi PV$.
%     Notice that \[||\Pi PV-\Pi P_{MV}||_{TV}\leq 2 P((X',V')\neq \Phi_V(X,V))\]
%     Notice that \[\begin{aligned}
%         P((X',V')\neq \Phi_V(X,V))&= P((X',V')\neq \Phi_V(X,V)\cap\{|\Phi_V(X,V)=\Phi_{MV}(X,V)\}) \\&+ P((X',V')\neq \Phi_V(X,V)\cap\{|\Phi_V(X,V)\neq\Phi_{MV}(X,V)\})
%         \\&\leq P((X',V')\neq \Phi_{MV}(X,V))+ P(\Phi_V(X,V)\neq\Phi_{MV}(X,V))
%     \end{aligned}\]
%     The first term can be handled now by accept/reject step, while the second can be bounded by $P(h_\lambda(X)\neq \nabla u(X))$ which is handled by the definition of the taming scheme.
%     Since \[P(h_\lambda(X)\neq \nabla u(X))\leq P(|X| \geq r_\lambda)\leq C_{l,\Pi,d} \lambda^3 \] and 
Couple one step of $P_V$ and one step of $P_{MV}$ by using the same
initial condition $(X,V)$ and the same uniform variable $W$. Noticing that 
\[
\|\Pi P_V-\Pi P_{MV}\|_{\TV}
\le
2\,P\!\left(
Z_{MV}\neq \Phi_V(X,V)
\right),
\]
where $Z_{MV}$ denotes the output of the Metropolized Verlet step.
Therefore,
\begin{equation}
\label{eq:coupling-split}
\|\Pi P_V-\Pi P_{MV}\|_{\TV}
\le
2\,\E[1-\alpha(X,V)]
+
2\,P\!\left(
\Phi_V(X,V)\neq \Phi_{MV}(X,V)
\right).
\end{equation}

We bound the two terms separately.

\medskip
\noindent
\textbf{Step 1: Difference between the tamed and untamed Verlet maps.}

Set
\[
    Y
    :=
    X+\lambda V-\frac{\lambda^2}{2}h(X).
\]
and $\rho_\lambda=r_\lambda-2$
If
\[
    |X|<\rho_\lambda
    \qquad\text{and}\qquad
    |Y|<\rho_\lambda,
\]
then $h_\lambda(X)=h(X)$. Hence the first position update of the
tamed Verlet map agrees with the untamed one and is equal to $Y$.
Since also $h_\lambda(Y)=h(Y)$, the second drift evaluation agrees as
well. Thus
\[
\left\{
\Phi_V(X,V)\neq \Phi_{MV}(X,V)
\right\}
\subset
\left\{|X|\ge\rho_\lambda\right\}
\cup
\left\{|Y|\ge\rho_\lambda\right\}.
\]
Consequently,
\begin{equation}
\label{eq:map-difference}
P\!\left(
\Phi_V(X,V)\neq \Phi_{MV}(X,V)
\right)
\le
P(|X|\ge\rho_\lambda)
+
P(|Y|\ge\rho_\lambda).
\end{equation}

To control the second probability, observe that
\[
    Y-X
    =
    \lambda V-\frac{\lambda^2}{2}h(X).
\]
Hence
\begin{align}
P(|Y|\ge\rho_\lambda)
&\le
P(|X|\ge\rho_\lambda-1)
+
P(|Y-X|\ge 1)
\nonumber\\
&\le
P(|X|\ge\rho_\lambda-1)
+
P\!\left(|V|\ge\frac{1}{2\lambda}\right)
+
P\!\left(|h(X)|\ge\lambda^{-2}\right).
\label{eq:Y-tail}
\end{align}

By strong log-concavity, $X$ has sub-Gaussian tails. Moreover,
$r_\lambda$ grows logarithmically in $\lambda^{-1}$, so the same
concentration argument used in the proof of the taming estimates
gives, for every fixed $q>0$,
\[
    P(|X|\ge\rho_\lambda-1)
    +
    P(|X|\ge\rho_\lambda)
    =
    O(\lambda^q).
\]
Also $V\sim N(0,I_d)$, so
\[
    P\!\left(|V|\ge\frac{1}{2\lambda}\right)
    =
    O\!\left(e^{-c/\lambda^2}\right).
\]
Finally, the local polynomial Lipschitz assumption implies polynomial
growth of $h$, and strong log-concavity of $\pi$ gives finite moments
of all orders. Hence, for any $p\ge 1$,
\[
    P\!\left(|h(X)|\ge\lambda^{-2}\right)
    \le
    \lambda^{2p}\E|h(X)|^p.
\]
Choosing $p$ sufficiently large, we obtain
\[
    P\!\left(
    \Phi_V(X,V)\neq \Phi_{MV}(X,V)
    \right)
    =
    O(\lambda^3).
\]
In particular,
\begin{equation}
\label{eq:map-difference-final}
    P\!\left(
    \Phi_V(X,V)\neq \Phi_{MV}(X,V)
    \right)
    \le C\lambda^3.
\end{equation}

\medskip
\noindent
    \[P((X',V')\neq \Phi_{MV}(X,V))\leq \E (1-a(X,V)\] one obtains the result.
    \begin{equation}\label{eq-TV-1}
        ||\Pi P_V-\Pi P_{MV}||_{TV}\leq \E (1-a(X,V)+  c\lambda^2.
    \end{equation}
    Notice that \[\E (1-a(X,V)\leq \E [H(\Phi_{MV}(X,V))-H(X,V)]_{+}\]
    For $x, v \in \mathbb{R}^d$, denoting $\xi=\lambda v-\lambda^2 / 2 \nabla U(x)$ and $\bar{\nabla} U=(\nabla U(x)+\nabla U(x+\xi)) / 2$,

\begin{equation}\label{eq-use}
\begin{aligned}
H\left(\Phi_{MV}(x, v)\right)-H(x, v) & =U(x+\xi)-U(x)+\frac{1}{2}|v-\lambda \bar{\nabla} U|^2-\frac{1}{2}|v|^2 \\
& =\xi \cdot  \int_0^1 \nabla U(x+s \xi) \mathrm{d} s-\lambda v \cdot \bar{\nabla} U+\lambda^2 / 2|\bar{\nabla} U|^2 \\
& =\xi \cdot\left(\int_0^1 \nabla U(x+s \xi) \mathrm{d} s-\bar{\nabla} U\right)+\lambda^2 / 2 \bar{\nabla} U \cdot(\bar{\nabla} U-\nabla U(x)) \\
& =\xi \cdot \int_0^1(2 g(s)-g(0)-g(1)) \mathrm{d} s+\lambda^2 / 2 \langle g(1) \cdot(g(1)-g(0)\rangle)
\end{aligned}
\end{equation}
with $g(s)=(\nabla U(x)+\nabla U(x+s \xi)) / 2$ for $s \in[0,1]$. Since $\nabla u$ is Locally Lipschitz, \[|2g(s)-g(1)-g(0)|\leq C(1+|x|+|\xi|)^l|\xi| \]
so it easy to see that \[\E |H\left(\Phi_{MV}(x, v)\right)-H(x, v)|\leq \E C (1+|x|+|\xi|)^l |\xi|^2 \leq C \lambda^2+ C'\lambda^3\]
which completes the proof since the moments of $|x|,|\xi|$ depend on moments of the invariant measure for which we have provided a uniform bound.
\end{proof}
\begin{corollary}\label{cor-TV pip^n-pi}
    We have \[||\Pi P^n -\Pi ||_{TV}\leq  \mathcal{O}(\lambda^2n).\]
\end{corollary}
\begin{proof}
    Notice that $\Pi$ is invariant under $P_O,$ so 
    \[||\Pi P-\Pi||_{TV}=||\Pi P_OP_VP_O-\Pi P_0||_{TV}=||\Pi P_V P_O-\Pi P_0||_{TV}\leq ||\Pi P_V-\Pi||\]
    where in the last step we have used that total variation decreases along any Markov operator.
    Now using the same argument \[\begin{aligned}
        ||\Pi P^n -\Pi ||_{TV}&\leq \sum_{k=1}^n ||(\Pi P-\Pi)P^{k-1}||_{TV}\\&
        \leq n||\Pi P- \Pi||_{TV}
        \\&\leq n ||\Pi P_V-\Pi||_{TV}
    \end{aligned}\]
    Applying the result of Lemma \ref{Lemma-comp with metrop} one obtains the result.
\end{proof}
\begin{proposition}[\cite{monmarche2021high}, Proposition 3, Adapted]\label{prop-TV to Wass}
    Let $\mu,\nu \in \mathcal{P}_2(\mathbb{R}^{2d}).$ We have  \[\begin{aligned}
&\|\nu P-\mu P\|_{T V} \leq \frac{K_2}{\lambda^{3 / 2}} \mathcal{W}_1(\nu, \mu),\\
&\text { with }\\
&K_2=\sqrt{\frac{\lambda}{2 \pi\left(1-\eta^2\right)}} \max \left(2+3 \lambda^2 M_\lambda / 2,3 \lambda\right)
\end{aligned}\]
\end{proposition}
\begin{theorem}
    We have \[||\mu_0 P^n-\Pi||_{TV}\leq \frac{K_2 c_{-}(a)c_{+}(a)}{\lambda^{3 / 2}} (1-\lambda \kappa)^{(n-1)/2} W_2(\mu,\Pi) +c \lambda^2 n.\]
\end{theorem}
\begin{proof}
    Writing \[||\mu P^n-\Pi||_{TV}\leq ||\mu P^n-\Pi P^n||_{TV}+ ||\Pi P^n-\Pi||_{TV}\]
    The second term can be bounded by Corollary \ref{cor-TV pip^n-pi}, while for the first term
     by Proposition \ref{prop-TV to Wass},
    \[\begin{aligned}
        ||\mu P^n-\Pi P^n||_{TV}&\leq  \frac{K_2}{\lambda^{3 / 2}} \mathcal{W}_1(\mu P^{n-1}, \Pi P^{n-1})
        \\&\leq \frac{K_2}{\lambda^{3 / 2}} \mathcal{W}_2(\mu P^{n-1}, \Pi P^{n-1})
        \\&\leq \frac{K_2c_{-}(a)c_{+}(a)}{\lambda^{3 / 2}} (1-\lambda \kappa)^{(n-1)/2} W_2(\mu,\Pi).
    \end{aligned}\]
\end{proof}
\subsection{Convergence of the OBABO scheme in TV under local Lipschitz Hessian assumption}
We will further assume that \[||\nabla^2 u(x)-\nabla^2u(y)||\leq L' (1+|x|+|y|)^{l'}|x-y|. \]
We shall use the second order property of the OBABO integrator to obtain an increased rate.

We shall follow the same strategy as in the case without the assumption of additional regularity. We will briefly state a few points where the analysis differs to provide the increased rate of convergence.
\begin{lemma}\label{increased-TV-Metropolis}
    Assume also that \[||\nabla^2 u(x)-\nabla^2u(y)||\leq L' (1+|x|+|y|)|x-y| \]
    Then \[||\Pi P_V-\Pi P_{MV}||_{TV}\leq \mathcal{O}(\lambda^3)\]
\end{lemma}
\begin{proof}
    The argument proceeds in the same way up to \eqref{eq-use}. Using the increased Hessian regularity,
    one writes 
    \[\begin{aligned}
2 \int_0^1 g(s) \mathrm{d} s-g(0)-g(1) & =2 \int_0^1 \int_0^s g^{\prime}(u) \mathrm{d} u \mathrm{~d} s-\int_0^1 g^{\prime}(s) \mathrm{d} s \\
& =\int_0^1(1-2 s) g^{\prime}(s) \mathrm{d} s \\
& =\int_0^1(1-2 s)\left(g^{\prime}(s)-g^{\prime}(0)\right) \mathrm{d} s .
\end{aligned}\]
By the chain rule, \[\begin{aligned}
|g'(s)-g'(0)|&=\frac{1}{2} |\nabla^2u(x+s\xi)\xi-\nabla^2u(x)\xi|\\&\leq \frac{1}{2}||\nabla^2u(x+s\xi)-\nabla^2u(x)|||\xi|\\&\leq \frac{L'}{2} (1+|x+s\xi|+|x|)^{l'}|s||\xi|^2\\&\leq \frac{L'}{2}(1+2|x|+|\xi|)^{l'} |\xi|^2\end{aligned}\] so
\[\E|\langle \xi,2 \int_0^1 g(s) \mathrm{d} s-g(0)-g(1)\rangle|\leq \E\frac{1}{2}L'(1+2|x|+|\xi|)^l |\xi|^3\leq C\lambda^3.\]
Continuing in the same spirit as \eqref{Lemma-comp with metrop} yields the result.
\end{proof}
\begin{theorem}
    We have \[||\mu_0 P^n-\Pi||_{TV}\leq \frac{K_2 c_{-}(a)c_{+}(a)}{\lambda^{3 / 2}} (1-\lambda \kappa)^{(n-1)/2} W_2(\mu,\Pi) +c \lambda^3 n.\]
\end{theorem}
\begin{proof}
    The proof follows as in the proof without the additional Lipschitz regularity by just replacing the result of Lemma \ref{Lemma-comp with metrop} with Lemma \ref{increased-TV-Metropolis}.
\end{proof}
\section{Acknowledgements}
{This work has been partially supported by project MIS 5154714 of the National Recovery and Resilience Plan Greece 2.0 funded by the European Union under the NextGenerationEU Program.}
\section{Declartion of AI usage}
During the preparation of the manuscript the authors have used ChatGPT 
to overcome the technical problem of non-gradient drift for the contraction estimates of the OBABO scheme and to simplify the proof for contraction of the Euler scheme. After using this tool/service, the author(s) reviewed and edited the content as needed and take(s) full responsibility for the content of the published article.
\appendix
\section{Appendix}
\begin{lemma}
Recall that $\eta=e^{-\lambda \gamma}$. Under our restrictions for $\gamma$ and $\lambda$,
    we have 
    \begin{equation}\label{eq-boundeta}
        \frac{\lambda \gamma}{2}\leq 1-\eta\leq \lambda\gamma
    \end{equation}
    \begin{equation}\label{eq-boundeta2}
0\leq-1+\eta +\lambda \gamma\leq \frac{\lambda ^2\gamma^2}{2}        
    \end{equation}
\end{lemma}

\begin{lemma}[ \cite{raginsky}, Lemma 3, adapted]\label{over-damp second moment}
Let ${L}_t$ be given by the overdamped Langevin SDE in \eqref{eq-over-Langevin} with initial condition $L_0$.
Then, under the strong convexity assumption we have
\begin{equation}\label{eq-oversecondm}
    \E |{L}_t|^2 \leq \E|L_0|^2e^{-mt} + 2\frac{d+m^{-1}|\nabla u(0)|^2}{m}(1-e^{-mt}).
\end{equation}
\end{lemma}
\begin{proof}
Notice that in the proof of Lemma 3 in \cite{raginsky} only the dissipativity condition is used which is a trivial consequence of strong convexity.
\end{proof}
\begin{lemma}\label{inv-measure bound}
    If $Y$ is a random variable such that $\mathcal{L}(Y)={\pi}$, then
\[\E |Y|^2\leq  \frac{2}{m} (|h(0)|^2m^{-1} +d). \]
As a result, \begin{equation}
    \E |Y|\leq \sqrt{ \frac{2}{m} (|h(0)|^2m^{-1} +d)} .
\end{equation}
\end{lemma}
\begin{proof}
Since the solution of the Langevin SDE \eqref{eq-over-Langevin} converges in $W_2$ distance to $\pi$ and\\ $\sup_{t\ge0}\E |\bar{L}_t|^2<\infty$ due to \eqref{eq-oversecondm}, this also implies the convergence of second moments. Therefore, if $Y$ is random variable such that $\mathcal{L}(W)=\pi$ by Lemma \ref{over-damp second moment} we have that
\[\E |Y|^2=\lim _{t\rightarrow \infty} \E |\bar{L}_t|^2\leq  \frac{2}{m} (|h(0)|^2m^{-1} +d) .\]
% Since the second marginal is a Gaussian the result follows easily as
% \[\E |Y|^2 \leq \frac{d}{\beta} + \E |W|^2\leq R + \frac{(R/b+1)d}{\beta}.\]
\end{proof}
\begin{lemma}
  \label{harge}[\cite{harge2004convex}, Theorem 1.1]
Let $X$ follow $\mathcal{N}(m_0,\Sigma)$ with density $\phi$, and let $Y$ have density $\phi f$ where $f$ is a log-concave function. Then for any convex map $g$ there holds:
\[\E [g\left(Y-\E Y\right)]\leq \E [g\left(X-\E X\right)] .\]
\end{lemma}
    \bibliography{referencenew.bib}

@inproceedings{altschuler2026shifted,
  title={Shifted composition IV: toward ballistic acceleration for log-concave sampling},
  author={Altschuler, Jason M and Chewi, Sinho and Zhang, Matthew S},
  booktitle={Proceedings of the 58th Annual ACM Symposium on Theory of Computing},
  pages={1739--1750},
  year={2026}
}

@article{paulin2026theoretical,
  title={Theoretical guarantees for stochastic gradient sampling methods via Gaussian convolution inequalities},
  author={Paulin, Daniel and Whalley, Peter A},
  journal={arXiv preprint arXiv:2604.24632},
  year={2026}
}

@article{schuh2024convergence,
  title={Convergence of kinetic Langevin samplers for non-convex potentials},
  author={Schuh, Katharina and Whalley, Peter A},
  journal={arXiv preprint arXiv:2405.09992},
  year={2024}
}

@article{Deniz,
  title={Nonasymptotic analysis of Stochastic Gradient Hamiltonian Monte Carlo under local conditions for nonconvex optimization},
  author={Akyildiz, {\"O}mer Deniz and Sabanis, Sotirios},
  journal={arXiv preprint arXiv:2002.05465},
  year={2020}
}

@article{TV,
  title={On explicit $ L\^{} 2$-convergence rate estimate for underdamped Langevin dynamics},
  author={Cao, Yu and Lu, Jianfeng and Wang, Lihan},
  journal={arXiv preprint arXiv:1908.04746},
  year={2019}
}

@article{dalalyan2018sampling,
  title={On sampling from a log-concave density using kinetic Langevin diffusions},
  author={Dalalyan, Arnak S and Riou-Durand, Lionel},
  journal={Bernoulli},
  volume={26},
  number={3},
  pages={1956--1988},
  year={2020},
  publisher={Bernoulli Society for Mathematical Statistics and Probability}
}

@article{Gao,
  title={Breaking reversibility accelerates Langevin dynamics for global non-convex optimization},
  author={Gao, Xuefeng and Gurbuzbalaban, Mert and Zhu, Lingjiong},
  journal={arXiv preprint arXiv:1812.07725},
  year={2018}
}

@article{dalalyan2017theoretical,
  title={Theoretical guarantees for approximate sampling from smooth and log-concave densities},
  author={Dalalyan, Arnak S},
  journal={Journal of the Royal Statistical Society: Series B (Statistical Methodology)},
  volume={79},
  number={3},
  pages={651--676},
  year={2017},
  publisher={Wiley Online Library}
}

@article{zhang2023nonasymptotic,
  title={Nonasymptotic estimates for stochastic gradient Langevin dynamics under local conditions in nonconvex optimization},
  author={Zhang, Ying and Akyildiz, {\"O}mer Deniz and Damoulas, Theodoros and Sabanis, Sotirios},
  journal={Applied Mathematics \& Optimization},
  volume={87},
  number={2},
  pages={25},
  year={2023},
  publisher={Springer}
}

@article{chau2019stochastic,
  title={Stochastic {G}radient {H}amiltonian {M}onte {C}arlo for {N}on-{C}onvex {L}earning},
  author={Chau, Huy N and Rasonyi, Miklos},
  journal={arXiv preprint arXiv:1903.10328},
  year={2019}
}

@inproceedings{cheng2018underdamped,
  title={Underdamped {L}angevin {MCMC}: {A} non-asymptotic analysis},
  author={Cheng, Xiang and Chatterji, Niladri S and Bartlett, Peter L and Jordan, Michael I},
  booktitle={Conference On Learning Theory},
  pages={300--323},
  year={2018}
}

@inproceedings{raginsky,
  title={Non-convex learning via {S}tochastic {G}radient {L}angevin {D}ynamics: a nonasymptotic analysis},
  author={Raginsky, Maxim and Rakhlin, Alexander and Telgarsky, Matus},
  booktitle={Conference on Learning Theory},
  pages={1674--1703},
  year={2017}
}

@article{durmus2017nonasymptotic,
  title={Nonasymptotic convergence analysis for the unadjusted {L}angevin algorithm},
  author={Durmus, Alain and Moulines, Eric},
  journal={The Annals of Applied Probability},
  volume={27},
  number={3},
  pages={1551--1587},
  year={2017},
  publisher={Institute of Mathematical Statistics}
}

@article{durmus2019high,
  title={High-dimensional {B}ayesian inference via the unadjusted {L}angevin algorithm},
  author={Durmus, Alain and Moulines, Eric},
  journal={Bernoulli},
  volume={25},
  number={4A},
  pages={2854--2882},
  year={2019},
  publisher={Bernoulli Society for Mathematical Statistics and Probability}
}

@article{convex,
  title={On stochastic gradient Langevin dynamics with dependent data streams in the logconcave case},
  author={Barkhagen, Mathias and Chau, Ngoc Huy and Moulines, {\'E}ric and R{\'a}sonyi, Mikl{\'o}s and Sabanis, Sotirios and Zhang, Ying},
  journal={Bernoulli},
  volume={27},
  number={1},
  pages={1--33},
  year={2021},
  publisher={Bernoulli Society for Mathematical Statistics and Probability}
}

@article{tula,
  title={The tamed unadjusted {L}angevin algorithm},
  author={Brosse, Nicolas and Durmus, Alain and Moulines, {\'E}ric and Sabanis, Sotirios},
  journal={Stochastic Processes and their Applications},
  volume={129},
  number={10},
  pages={3638--3663},
  year={2019},
  publisher={Elsevier}
}

@article{nonconvex,
  title={On Stochastic Gradient Langevin Dynamics with Dependent Data Streams: The Fully Nonconvex Case},
  author={Chau, Ngoc Huy and Moulines, {\'E}ric and R{\'a}sonyi, Miklos and Sabanis, Sotirios and Zhang, Ying},
  journal={SIAM Journal on Mathematics of Data Science},
  volume={3},
  number={3},
  pages={959--986},
  year={2021},
  publisher={SIAM}
}

@article{berkeley,
  title={Sharp convergence rates for {L}angevin dynamics in the nonconvex setting},
  author={Cheng, Xiang and Chatterji, Niladri S and Abbasi-Yadkori, Yasin and Bartlett, Peter L and Jordan, Michael I},
  journal={arXiv preprint arXiv:1805.01648},
  year={2018}
}

@article{hola,
  title={Higher order {L}angevin {M}onte {C}arlo algorithm},
  author={Sabanis, Sotirios and Zhang, Ying},
  journal={Electronic Journal of Statistics},
  volume={13},
  number={2},
  pages={3805--3850},
  year={2019},
  publisher={The Institute of Mathematical Statistics and the Bernoulli Society}
}

@article{TUSLA,
  title={Taming neural networks with tusla: Nonconvex learning via adaptive stochastic gradient langevin algorithms},
  author={Lovas, Attila and Lytras, Iosif and R{\'a}sonyi, Mikl{\'o}s and Sabanis, Sotirios},
  journal={SIAM Journal on Mathematics of Data Science},
  volume={5},
  number={2},
  pages={323--345},
  year={2023},
  publisher={SIAM}
}

@article{johnston2024strongly,
  title={A strongly monotonic polygonal euler scheme},
  author={Johnston, Tim and Sabanis, Sotirios},
  journal={Journal of Complexity},
  volume={80},
  pages={101801},
  year={2024},
  publisher={Elsevier}
}

@article{gao2021global,
  title={Global Convergence of Stochastic Gradient Hamiltonian Monte Carlo for Nonconvex Stochastic Optimization: Nonasymptotic Performance Bounds and Momentum-Based Acceleration},
  author={Gao, Xuefeng and G{\"u}rb{\"u}zbalaban, Mert and Zhu, Lingjiong},
  journal={Operations Research},
  year={2021},
  publisher={INFORMS}
}

@article{harge2004convex,
  title={A convex/log-concave correlation inequality for Gaussian measure and an application to abstract Wiener spaces},
  author={Harg{\'e}, Gilles},
  journal={Probability theory and related fields},
  volume={130},
  number={3},
  pages={415--440},
  year={2004},
  publisher={Springer}
}

@article{neufeld2022non,
  title={Non-asymptotic convergence bounds for modified tamed unadjusted Langevin algorithm in non-convex setting},
  author={Neufeld, Ariel and En, Matthew Ng Cheng and Zhang, Ying},
  journal={arXiv preprint arXiv:2207.02600},
  year={2022}
}

@article{mou2022improved,
  title={Improved bounds for discretization of Langevin diffusions: Near-optimal rates without convexity},
  author={Mou, Wenlong and Flammarion, Nicolas and Wainwright, Martin J and Bartlett, Peter L},
  journal={Bernoulli},
  volume={28},
  number={3},
  pages={1577--1601},
  year={2022},
  publisher={Bernoulli Society for Mathematical Statistics and Probability}
}

@article{vempala2019rapid,
  title={Rapid convergence of the unadjusted langevin algorithm: Isoperimetry suffices},
  author={Vempala, Santosh and Wibisono, Andre},
  journal={Advances in neural information processing systems},
  volume={32},
  year={2019}
}

@article{nguyen2021unadjusted,
  title={Unadjusted Langevin algorithm for non-convex weakly smooth potentials},
  author={Nguyen, Dao and Dang, Xin and Chen, Yixin},
  journal={arXiv preprint arXiv:2101.06369},
  year={2021}
}

@inproceedings{erdogdu2021convergence,
  title={On the convergence of langevin monte carlo: The interplay between tail growth and smoothness},
  author={Erdogdu, Murat A and Hosseinzadeh, Rasa},
  booktitle={Conference on Learning Theory},
  pages={1776--1822},
  year={2021},
  organization={PMLR}
}

@article{majka2020nonasymptotic,
  title={Nonasymptotic bounds for sampling algorithms without log-concavity},
  author={Majka, Mateusz B and Mijatovi{\'c}, Aleksandar and Szpruch, {\L}ukasz},
  journal={The Annals of Applied Probability},
  volume={30},
  number={4},
  pages={1534-1581},
  year={2020},
  
}

@inproceedings{balasubramanian2022towards,
  title={Towards a theory of non-log-concave sampling: first-order stationarity guarantees for langevin monte carlo},
  author={Balasubramanian, Krishna and Chewi, Sinho and Erdogdu, Murat A and Salim, Adil and Zhang, Shunshi},
  booktitle={Conference on Learning Theory},
  pages={2896--2923},
  year={2022},
  organization={PMLR}
}

@inproceedings{erdogdu2022convergence,
  title={Convergence of Langevin Monte Carlo in chi-squared and R{\'e}nyi divergence},
  author={Erdogdu, Murat A and Hosseinzadeh, Rasa and Zhang, Shunshi},
  booktitle={International Conference on Artificial Intelligence and Statistics},
  pages={8151--8175},
  year={2022},
  organization={PMLR}
}

@article{zhang2023improved,
  title={Improved discretization analysis for underdamped Langevin Monte Carlo},
  author={Zhang, Matthew and Chewi, Sinho and Li, Mufan Bill and Balasubramanian, Krishnakumar and Erdogdu, Murat A},
  journal={arXiv preprint arXiv:2302.08049},
  year={2023}
}

@article{monmarche2021high,
  title={High-dimensional MCMC with a standard splitting scheme for the underdamped Langevin diffusion.},
  author={Monmarch{\'e}, Pierre},
  journal={Electronic Journal of Statistics},
  volume={15},
  number={2},
  pages={4117--4166},
  year={2021},
  publisher={The Institute of Mathematical Statistics and the Bernoulli Society}
}

@article{10.3150/20-BEJ1297,
author = {Yi-An Ma and Niladri S. Chatterji and Xiang Cheng and Nicolas Flammarion and Peter L. Bartlett and Michael I. Jordan},
title = {{Is there an analog of Nesterov acceleration for gradient-based MCMC?}},
volume = {27},
journal = {Bernoulli},
number = {3},
publisher = {Bernoulli Society for Mathematical Statistics and Probability},
pages = {1942 -- 1992},
year = {2021},
doi = {10.3150/20-BEJ1297},
URL = {https://doi.org/10.3150/20-BEJ1297}
}

@inproceedings{nesterov1983method,
  title={A method for unconstrained convex minimization problem with the rate of convergence O (1/k2)},
  author={Nesterov, Yiri},
  booktitle={Dokl. Akad. Nauk. SSSR},
  volume={269},
  pages={543},
  year={1983}
}

@article{chewi2024ballistic,
  title={The ballistic limit of the log-Sobolev constant equals the Polyak-$\{$$\backslash$L$\}$ ojasiewicz constant},
  author={Chewi, Sinho and Stromme, Austin J},
  journal={arXiv preprint arXiv:2411.11415},
  year={2024}
}

@article{altschuler2024faster,
  title={Faster high-accuracy log-concave sampling via algorithmic warm starts},
  author={Altschuler, Jason M and Chewi, Sinho},
  journal={Journal of the ACM},
  volume={71},
  number={3},
  pages={1--55},
  year={2024},
  publisher={ACM New York, NY}
}

@article{leimkuhler2024contraction,
  title={Contraction and convergence rates for discretized kinetic Langevin dynamics},
  author={Leimkuhler, Benedict J and Paulin, Daniel and Whalley, Peter A},
  journal={SIAM Journal on Numerical Analysis},
  volume={62},
  number={3},
  pages={1226--1258},
  year={2024},
  publisher={SIAM}
}

@article{lytras2024tamed,
  title={Tamed Langevin sampling under weaker conditions},
  author={Lytras, Iosif and Mertikopoulos, Panayotis},
  journal={arXiv preprint arXiv:2405.17693},
  year={2024}
}

@article{kuwada2010duality,
  title={Duality on gradient estimates and Wasserstein controls},
  author={Kuwada, Kazumasa},
  journal={Journal of Functional Analysis},
  volume={258},
  number={11},
  pages={3758--3774},
  year={2010},
  publisher={Elsevier}
}

@article{lytras2023taming,
  title={Taming under isoperimetry},
  author={Lytras, I. and Sabanis, S.},
  journal={Stochastic Process. Appl.},
  volume={188},
  year={2025}
}

@article{johnston2023kinetic,
  title={Kinetic Langevin MCMC sampling without gradient Lipschitz continuity—the strongly convex case},
  author={Johnston, T. and Lytras, I. and Sabanis, S.},
  journal={J. Complex.},
  volume={85},
  year={2024}
}

@InProceedings{pmlr-v258-lytras25a,
  title = 	 {Tamed Langevin sampling under weaker conditions},
  author =       {Lytras, Iosif and Mertikopoulos, Panayotis},
  booktitle = 	 {Proceedings of The 28th International Conference on Artificial Intelligence and Statistics},
  pages = 	 {847--855},
  year = 	 {2025},
  volume = 	 {258},
  series = 	 {Proceedings of Machine Learning Research},
  month = 	 {03--05 May},
  publisher =    {PMLR}
}
\end{document}